\documentclass[graybox]{svmult}

\usepackage{mathptmx}       
\usepackage{helvet}         
\usepackage{courier}        
\usepackage{type1cm}        
\usepackage{makeidx}         
\usepackage{graphicx}        
\usepackage{multicol}        
\usepackage[bottom]{footmisc}
\usepackage{verbatim}		
\usepackage{epstopdf}		
\DeclareGraphicsExtensions{.eps}
\usepackage{amssymb}

\makeindex             

\spnewtheorem*{prerequisites}{Suggested prerequisites}{\bf}{\small\em}

\spdefaulttheorem{project}{Research Project}{\bf}{}
\newcommand{\resproject}[1]{\begin{svgraybox} \begin{project} #1 \end{project} \end{svgraybox}}

\spdefaulttheorem{cproblem}{Challenge Problem}{\bf}{}

\begin{document}

\title*{Steady and Stable: Numerical Investigations of Nonlinear Partial Differential Equations}
\titlerunning{Numerical Investigations of Nonlinear PDEs}
\author{R. Corban Harwood}
\institute{R. Corban Harwood \at George Fox University, 414 North Meridian Street, Newberg, OR 97132, USA, \email{rharwood@georgefox.edu}}

%
%
\maketitle
\global\long\def\dx{\Delta x}
\global\long\def\dt{\Delta t}

\abstract*{A nonlinear partial differential equation is a nonlinear relationship between an unknown function and how it changes due to two or more input variables. A numerical method reduces such an equation to arithmetic for quick visualization, but this can be misleading if the method is not developed and operated carefully due to numerical oscillations, instabilities, and other artifacts of the programmed solution. This chapter provides a friendly introduction to error analysis of numerical methods for nonlinear partial differential equations along with modern approaches to visually demonstrate analytical results with confidence. These techniques are illustrated through several detailed examples with a relevant governing equation, and lead to open questions suggested for undergraduate-accessible research.}

\abstract{ 
A nonlinear partial differential equation is a nonlinear relationship between an unknown function and how it changes due to two or more input variables. A numerical method reduces such an equation to arithmetic for quick visualization, but this can be misleading if the method is not developed and operated carefully due to numerical oscillations, instabilities, and other artifacts of the programmed solution. This chapter provides a friendly introduction to error analysis of numerical methods for nonlinear partial differential equations along with modern approaches to visually demonstrate analytical results with confidence. These techniques are illustrated through several detailed examples with a relevant governing equation, and lead to open questions suggested for undergraduate-accessible research.}

\begin{prerequisites}
differential equations, linear algebra, some programming experience
\end{prerequisites}

\section{Introduction}
\label{sec:intro}

Mathematics is a language which can describe patterns in everyday life as well as abstract concepts existing only in our minds. Patterns exist in data, functions, and sets constructed around a common theme, but the most tangible patterns are visual. Visual demonstrations can help undergraduate students connect to abstract concepts in advanced mathematical courses. The study of partial differential equations, in particular, benefits from numerical analysis and simulation. 

Applications of mathematical concepts are also rich sources of visual aids to gain perspective and understanding. Differential equations are a natural way to model relationships between different measurable phenomena. Do you wish to predict the future behavior of a phenomenon? Differential equations do just that--after being developed to adequately match the dynamics involved. For instance, say you are interested in how fast different parts of a frying pan heat up on the stove. Derived from simplifying assumptions about density, specific heat, and the conservation of energy, the heat equation will do just that! In section \ref{sub:pdes} we use the heat equation (\ref{eqn:LinearTestIBVP}), called a test equation, as a control in our investigations of more complicated partial differential equations (PDEs). To clearly see our predictions of the future behavior, we will utilize numerical methods to encode the dynamics modeled by the PDE into a program which does basic arithmetic to approximate the underlying calculus.
To be confident in our predictions, however, we need to make sure our numerical method is developed in a way that keeps errors small for better accuracy. Since the method will compound error upon error at every iteration, the method must manage how the total error grows in a stable fashion. Else, the computed values will ``blow up'' towards infinity--becoming nonnumerical values once they have exceeded the largest number the computer can store. Such instabilities are adamantly avoided in commercial simulations using {\it adaptive} methods, such as the Rosenbrock method implemented as {\tt ode23s} in MATLAB \cite{Moler2008}. These adaptive methods reduce the step size as needed to ensure stability, but in turn increase the number of steps required for your prediction. Section \ref{sec:NumericalPDEs} gives an overview of numerical partial differential equations. Burden \cite{Burden2011} and Thomas \cite{Thomas1995} provide great beginner and intermediate introductions to the topic, respectively. In section \ref{sec:analysis} we compare basic and adaptive methods in verifying accuracy, analyzing stability through fundamental definitions and theorems, and finish by tracking oscillations in solutions. Researchers have developed many ways to reduce the effect of numerically induced oscillations which can make solutions appear infeasible \cite{Britz2003, Osterby2003}. Though much work has been done in studying the nature of numerical oscillations in ordinary differential equations \cite{CGraham2016, Gao2015}, some researchers have applied this investigation to nonlinear evolution PDEs \cite{Harwood2011, Lakoba2016P1}. Recently, others have looked at the stability of steady-state and traveling wave solutions to nonlinear PDEs \cite{HarleyMarangell2015,LeVeque2007,Nadin2011}, with more work to be done. We utilize these methods in our parameter analysis in section \ref{sec:parameters} and set up several project ideas for further research. Undergraduate students have recently published related work, for example, in steady-state and stability analysis \cite{Aron2014, Sarra2011} and other numerical investigations of PDEs \cite{Juhnke2015}.

\section{Numerical Differential Equations}
\label{sec:NumericalPDEs}

In applying mathematics to real-world problems, a differential equation can encode information about how a quantity changes in time or space relative to itself more easily than forming the function directly by fitting the data. The mass of a bacteria colony is such a quantity. In this example, tracking the intervals over which the population's mass doubles can be related to measurements of the population's mass to find its exponential growth function. Differential equations are formed from such relationships. Finding the pattern of this relationship allows us to solve the differential equation for the function we seek. This pattern may be visible in the algebra of the function, but can be even more clear in graphs of numerical solutions.


\subsection{Overview of Differential Equations}
\label{sub:pdes}

An ordinary differential equation (ODE) is an equation involving derivatives of a single variable whose solution is a function which satisfies the given relationship between the function and its derivatives. Because the integration needed to undo each derivative introduces a constant of integration, conditions are added for each derivative to specify a single function. The order of a differential equation is the highest derivative in the equation. Thus, a first order ODE needs one condition while a third order ODE needs three. 
\begin{definition}\label{def:IVP}
An initial value problem (IVP) with a first order ODE is defined as
\begin{eqnarray}\label{eqn:firstode}
 \frac{dx}{dt} &=& f(x,t)\\\nonumber
 x(t_0)&=&x_0,
\end{eqnarray}
where $t$ is the independent variable, $x\equiv x(t)$ is the dependent variable (also called the unknown function) with initial value of $x(t_0)=x_0$, and $f(x,t)$ is the {\it slope function}.
\end{definition}



As relationships between a function and its derivatives, a PDE and an ODE are much alike. Yet PDEs involve multivariable functions and each derivative is a partial derivative in terms of one or more independent variables. Recall that a partial derivative focuses solely on one variable when computing derivatives. For example, $\frac{\partial}{\partial t} e^{-2t}\sin(3x)=-2e^{-2t}\sin(3x)$. Similar to the ways ordinary derivatives are notated 
, partial derivatives can be written in operator form or abbreviated with subscripts (e.g. $\frac{\partial^2 u}{\partial x^2} = u_{xx}$).
{\it Linear} PDEs are composed of a sum of scalar multiples of the unknown function, its derivatives, as well as functions of the independent variables. A PDE is {\it nonlinear} when it has term which is not a scalar multiple of an unknown, such as $\rho u(1-u)$ in (\ref{eqn:FisherIBVP}) or an arbitrary function of the unknown. To introduce problems involving PDEs, we begin with the simplest type of boundary conditions, named after mathematician Peter Dirichlet (1805-1859), and a restriction to first order in time (called evolution PDEs). Note that the number of conditions needed for a unique solution to a PDE is the total of the orders in each independent variable \cite{Thomas1995}. Sufficient number of conditions, however, does not prove uniqueness. The maximum principle and energy method are two ways uniqueness of a solution can be proven \cite{Evans2010}, but such analysis is beyond the scope of this chapter. 

\begin{definition}\label{def:IBVP}
An initial boundary value problem (IBVP) with a first order (in time) evolution PDE with Dirichlet boundary conditions is defined as
\begin{eqnarray}\label{eqn:IBVP}
 \frac{\partial u}{\partial t} &=& f\left(x,t, u,\frac{\partial u}{\partial x},\frac{\partial^2 u}{\partial x^2},...\right)\\ \nonumber
u(x,0) = u_0(x),\\ \nonumber
u(0,t) = a\\ \nonumber
u(L,t) = b
\end{eqnarray}
where $x,t$ are the independent variables, $u$ is the dependent variable (also called the unknown function) with initial value of $u(x,t)=u_0(x)$ and boundary values $u=a,b$ whenever $x=a,b$ respectively, and $f$ can be any combination of independent variables and any spatial partials of the dependent variable.
\end{definition}

\begin{example}
Let us analyze the components of the following initial boundary value problem:
\begin{eqnarray}\label{eqn:FisherIBVP}
u_t = \delta u_{xx} + \rho u(1-u),\\\nonumber
u(x,0) = u_0(x),\\\nonumber
u(0,t) = 0\\\nonumber
u(10,t) = 1
\end{eqnarray}
First, the PDE is nonlinear due to the $u(1-u)$ term. Second, the single initial condition matches the 1st order in time ($u_t$) and the two boundary values match the 2nd order in space ($u_{xx}$). Thus, this IBVP has sufficient number of conditions needed for a unique solution which supports but does not prove uniquess. Third, parameters $\delta,\rho$ and the initial profile function $u_0(x)$ are kept unspecified. 
\end{example}

This reaction-diffusion equation is known as the Fisher-KPP equation for the four mathematicians who all provided great analytical insight into it: Ronald Fisher (1890-1962), Andrey Kolmogorov (1903-1987), Ivan Petrovsky (1901-1973), and Nikolaj Piscounov (1908-1977) \cite{Fisher1937, Kolmogorov1991}. Though it is more generally defined as an IVP, in this chapter we study it in its simpler IBVP form. Coeffients $\delta,\rho$ represent the diffusion and reaction rates and varying their values lead to many interesting behaviors. The Fisher-KPP equation models how a quantity switches between phases, such as genes switching to advantageous alleles where it was originally studied \cite{Fisher1937}. 

The form of the initial condition function, $u_0(x)$, is kept vague due to the breadth of physically meaning and theoretically interesting functions which could initialize our problem. Thus, we will use a polynomial fitting functions {\tt polyfit()} and {\tt polyval()} in example code {\tt PDE\_Analysis\_Setup.m} to set up a polynomial of any degree which best goes through the boundary points and other provided points. This description of the initial condition allows us to explore functions constrained by their shape within the bounds of the equilibrium point $\bar{u}$ analyzed in section \ref{sub:steadystate}.

\begin{figure}[t]
\centerline{
\includegraphics[scale=.45]{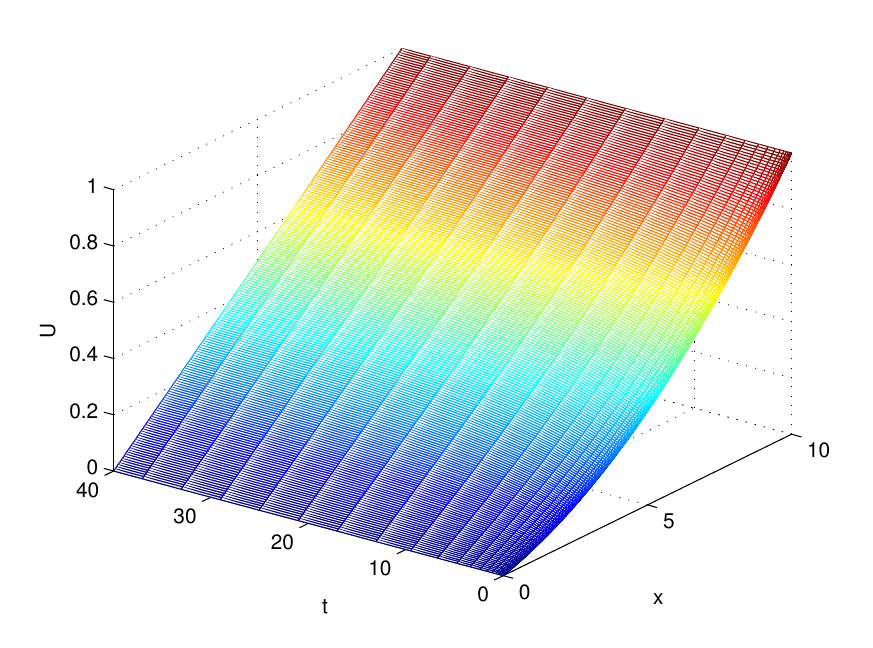} 
\includegraphics[scale=.45]{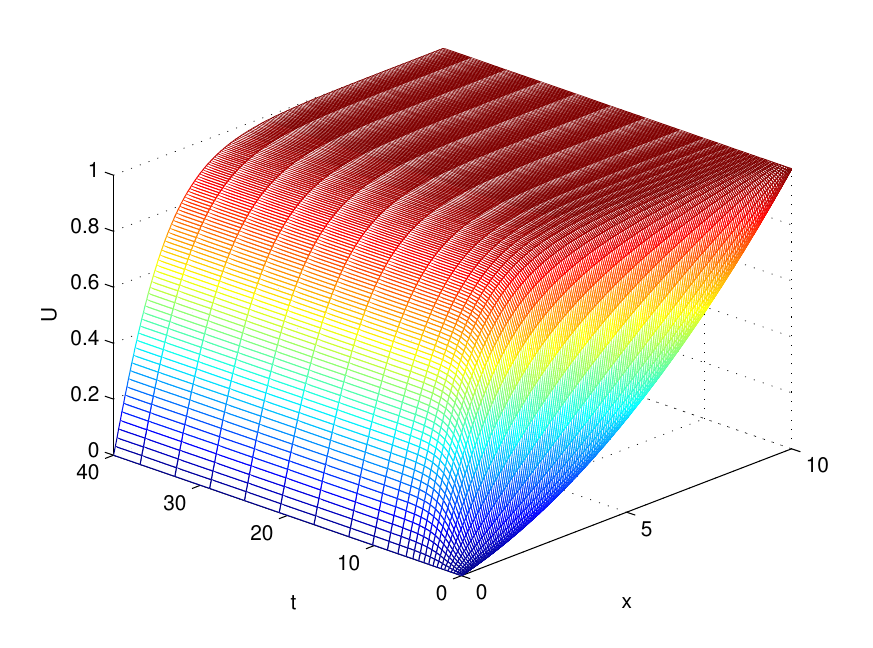} 
}
\caption{Comparison of numerical solutions using the adaptive Rosenbrock method ({\tt ode23s} in MATLAB) for the (left) linear Test equation (\ref{eqn:LinearTestIBVP}) using $\rho=0$ and (right) Fisher-KPP equation (\ref{eqn:FisherIBVP})  using $\rho=1$, where all other parameters use the default values of $a=0, b=1, L=10, \delta=1, \dx=0.05, {\rm degree}=2, c=\frac{1}{3}$ and initial condition from line 8 in {\tt PDE\_Analysis\_Setup.m} found in Appendix \ref{App:Setup}.}
\label{fig:lf5ode23s}
\end{figure}

\begin{exercise}
Consider the PDE 
\begin{equation}\label{eqn:ex:PDE}
u_t = 4u_{xx}.
\end{equation} 
\renewcommand{\labelenumi}{\alph{enumi})}
\begin{enumerate}
\item Determine the order in time and space and how many initial and boundary conditions are needed to define a unique solution. 
\item Using Definition (\ref{def:IBVP}) as a guide, write out the IBVP for an unknown $u(x,t)$ such that it has an initial profile of $\sin(x)$, boundary value of $0$ whenever $x=0$ and $x=\pi$, and is defined for $0\leq x\leq 1,t\geq0$. 
\item Verify that you have enough initial and boundary conditions as determined previously.
\item Verify that the function,
\begin{equation}
u(x,t)=e^{-4t}\sin(x),
\end{equation}
is a solution to equation (\ref{eqn:ex:PDE}) by evaluating both sides of the PDE and checking the initial and boundary conditions.
\end{enumerate}
\renewcommand{\labelenumi}{\arabic{enumi}.}
\end{exercise}


Error is more difficult to analyze for nonlinear PDEs, so it is helpful to have an associated linear version of your equation to analyze first. We will compare our analysis of reaction-diffusion equations to the Test equation,
\begin{eqnarray}\label{eqn:LinearTestIBVP}
u_t = \delta u_{xx},\\\nonumber
u(x,0) = u_0(x),\\\nonumber
u(0,t) = 0\\\nonumber
u(10,t) = 1
\end{eqnarray}
which is the heat equation in one dimension with constant heat forced at the two end points \cite{Burden2011}. Note, this is not a direct linearization of the Fisher-KPP equation (\ref{eqn:FisherIBVP}), but it behaves similarly for large values of $\delta$. Figure \ref{fig:lf5ode23s} provides a comparison of the solutions to the linear Test equation (\ref{eqn:LinearTestIBVP}) and the Fisher-KPP equation (\ref{eqn:FisherIBVP}) for $\delta=1$. Note how the step size in time for both solutions increases dramatically to have large, even spacing as the solution nears the steady-state solution. Adaptive methods, like MATLAB's ode23s, adjust to a larger step size as the change in the solution diminishes.

\subsection{Overview of Numerical Methods}
\label{sub:methods}
Numerical methods are algorithms which solve problems using arithmetic computations instead of algebraic formulas. They provide quick visualizations and approximations of solutions to problems which are difficult or less helpful to solve exactly.

Numerical methods for differential equations began with methods for approximating integrals: starting with left and right Riemann sums, then progressing to the trapezoidal rule, Simpson's rule and others to increase accuracy more and more efficiently. Unfortunately, the value of the slope function for an ODE is often unknown so such approximations require modifications, such as Taylor series expansions for example, to predict and correct slope estimates. Such methods for ODEs can be directly applied to evolution PDEs (\ref{eqn:IBVP}). Discretizing in space, we create a system of ordinary differential equations with vector ${\bf U}(t)$ with components $U_m(t)$ approximating the unknown function $u(x,t)$ at discrete points $x_m$. The coefficients of linear terms are grouped into matrix $D(t)$ and nonlinear terms are left in vector function ${\bf R}(t,{\bf U})$. In the following analysis, we will assume that $t$ is not explicit in the matrix $D$ or nonlinear vector ${\bf R}({\bf U})$ to obtain the general form of a reaction-diffusion model (\ref{eqn:IBVP})
\begin{equation}\label{eqn:generalRD}
\frac{d {\bf U}}{dt} = D {\bf U} + {\bf R}({\bf U})+{\bf B}.
\end{equation}

\begin{example}\label{ex:discretization}
We will discretize the Fisher-KPP equation (\ref{eqn:FisherIBVP}) in space using default parameter values $a=0, b=1, L=10, \delta=1, \dx=0.05, \rho=1$ in {\tt PDE\_Analysis\_ Setup.m} found in Appendix \ref{App:Setup}. See Figure (\ref{fig:lf5ode23s}) (right) for the graph. Evenly dividing the interval $[0,10]$ with $\dx=0.05=\frac{1}{20}$ results in 199 spatial points, $x_m$, where the function is unknown (plus the two end points where it is known: $U_0=a,U_{200}=b$). Using a centered difference approximation of $U_{xx}$ \cite{Burden2011},
\begin{eqnarray}
\left( U_{xx}\right)_1 &\approx & \frac{a-2U_1+U_{2}}{\dx^2},\\ \nonumber
\left( U_{xx}\right)_m &\approx & \frac{U_{m-1}-2U_m+U_{m+1}}{\dx^2},\ 2\leq m\leq 198,\\ \nonumber
\left( U_{xx}\right)_{199} &\approx & \frac{U_{198}-2U_{199}+b}{\dx^2},
\end{eqnarray}
the discretization of (\ref{eqn:FisherIBVP}) can be written as
\begin{eqnarray}
\frac{d {\bf U}}{dt} &=& D {\bf U} + {\bf R}({\bf U})+{\bf B},\\ \nonumber
D &=& \frac{\delta}{\dx^2}	\left[ \begin{array}{cccc}
	-2 & 1 &... &0\\
	1 & -2 & \ddots &... \\
	...&\ddots &\ddots &1\\
	0 &...& 1 & -2
	\end{array}\right]\\ \nonumber
{\bf R}({\bf U}) &=& 	\rho \left[ \begin{array}{c}
	U_1(1-U_1)\\
	...\\
	U_{199}(1-U_{199})\\
	\end{array}\right] = \rho\left(I-{\rm diag}({\bf U})\right){\bf U}\\  \nonumber
{\bf B} &=& 	\frac{\delta}{\dx^2}\left[ \begin{array}{c}
	a\\
	0\\
	...\\
	0\\
	b\\
	\end{array}\right]
\end{eqnarray}
with a tridiagonal matrix $D$, a nonlinear vector function ${\bf R}$ which can be written as a matrix product using diagonal matrix formed from a vector (diag()), and a sparse constant vector ${\bf B}$ which collects the boundary information.
\end{example}

By the Fundamental Theorem of Calculus \cite{Burden2011}, the exact solution to (\ref{eqn:generalRD}) over a small interval of time $\dt$ is found by integration from $t_n$ to $t_{n+1}=t_n+\dt$ as
\begin{equation}\label{eqn:numerical-integral}
{\bf U}^{n+1} = {\bf U}^n + \int_{t_n}^{t_{n+1}}\left(D {\bf U}(t) + {\bf R}({\bf U(t)})+{\bf B}\right)dt,
\end{equation}
where each component $U_m(t)$ has been discretized in time to create an array of components $U_m^n$ approximating the solution $u(x_m,t_n)$. 
Note that having the unknown function ${\bf U}(t)$ inside the integral (\ref{eqn:numerical-integral}) makes it impossible to integrate exactly, so we must approximate. Approximating with a left Riemann sum results in the Forward Euler (a.k.a. classic Euler) method \cite{Moler2008},
\begin{equation}\label{eqn:FWDEuler}
{\bf U}^{n+1} = {\bf U}^n + \dt\left(D( {\bf U}^n + {\bf R}({\bf U}^n)+{\bf B}\right),
\end{equation}
while approximating with a right Riemann sum results in the Backward Euler method \cite{Moler2008}
\begin{equation}\label{eqn:BWDEuler}
{\bf U}^{n+1} = {\bf U}^n + \dt\left(D {\bf U}^{n+1} + {\bf R}({\bf U}^{n+1})+{\bf B}\right).
\end{equation}

Although approximating integrals with left and right Riemann sums is a similar task, in solving differential equations, they can be very different. Forward Euler (\ref{eqn:FWDEuler}) is referred to as an {\it explicit} method since the unknown ${\bf U}^{n+1}$ can be directly computed in terms of known quantities such as the current known approximation ${\bf U}^n$, while Backward Euler  (\ref{eqn:BWDEuler}) is referred to as an {\it implicit} method since the unknown ${\bf U}^{n+1}$ is solved in terms of both known ${\bf U}^n$ and unknown ${\bf U}^{n+1}$ quantities. Explicit methods are simple to set up and compute, while implicit methods may not be solvable at all. If we set ${\bf R}({\bf U})\equiv {\bf 0}$ to make equation (\ref{eqn:generalRD}) linear, then an implicit method can be easily written in the explicit form, as shown in the Example \ref{ex:BWD}. Otherwise, an unsolvable implicit method can be approximated with a numerical root-finding method such as Newton's method (\ref{eqn:Newton}) which is discussed in section \ref{sub:Newtonmethod}, but nesting numerical methods is much less efficient than implementing an explicit method as it employs a truncated Taylor series to mathematically approximate the unknown terms. The main reasons to use implicit methods are for stability, addressed in section \ref{sub:stability}. The following examples demonstrate how to form the two-level matrix form. 
\begin{definition}\label{def:two-level}
A two-level numerical method for an evolution equation (\ref{eqn:IBVP}) is an iteration which can be written in the two-level matrix form
\begin{equation}\label{eqn:two-level_method}
{\bf U}^{n+1}=M\ {\bf U}^{n}+{\bf N},
\end{equation}
where $M$ is the combined transformation matrix and ${\bf N}$ is the
resultant vector. Note, both $M$ and ${\bf N}$ may update every iteration, especially when the PDE is nonlinear, but for many basic problems, $M$ and ${\bf N}$ will be constant.
\end{definition}


\begin{example}\label{ex:FWD}(Forward Euler)
Determine the two-level matrix form for the Forward Euler method for the Fisher-KPP equation (\ref{eqn:FisherIBVP}). Since Forward Euler is already explicit, we simply factor out the ${\bf U}^n$ components from equation (\ref{eqn:FWDEuler}) to form
\begin{eqnarray}
{\bf U}^{n+1} &=& {\bf U}^n + \dt\left(D( {\bf U}^n + \rho \left(I-{\rm diag}({\bf U}^n)\right){\bf U}^n+{\bf B}\right),\\ \nonumber
&=& M{\bf U}^n + {\bf N},\\ \nonumber
M &=& \left(I + \dt D +\dt\rho \left(I-{\rm diag}({\bf U}^n)\right)\right),\\ \nonumber
{\bf N} &=& \dt {\bf B},
\end{eqnarray}
where $I$ is the identity matrix, ${\bf N}$ is constant, and $M$ updates with each iteration since it depends on ${\bf U}^n$.
\end{example}

\begin{example}\label{ex:BWD}(Linear Backward Euler)
Determine the two-level matrix form for the Backward Euler method for the Test equation (\ref{eqn:LinearTestIBVP}). In the Backward Euler method (\ref{eqn:BWDEuler}), the unknown ${\bf U}^{n+1}$ terms can be grouped and the coefficient matrix $I-\dt D$ inverted to write it explicitly as
\begin{eqnarray}
{\bf U}^{n+1} &=&  M{\bf U}^n+{\bf N},\\ \nonumber
M &=& \left(I - \dt D \right)^{-1}\\ \nonumber
N &=& \dt\left(I - \dt D \right)^{-1}{\bf B}
\end{eqnarray}
where the method matrix $M$ and additional vector ${\bf N}$ are constant.
\end{example}

Just as the trapezoid rule takes the average of the left and right Riemann sums, the Crank-Nicolson method (\ref{eqn:CN}) averages the Forward and Backward Euler methods \cite{CrankNicolson1947}.
\begin{equation}\label{eqn:CN}
{\bf U}^{n+1} = {\bf U}^n + \frac{\dt}{2} D\left( {\bf U}^{n} +  {\bf U}^{n+1}\right) + \frac{\dt}{2}\left({\bf R}({\bf U}^{n}) + {\bf R}({\bf U}^{n+1})\right)+\dt{\bf B}.
\end{equation}
One way to truncate an implicit method of a nonlinear equation into an explicit method is called a {\it semi-implicit} method \cite{Burden2011}, which treats the nonlinearity as known information (evaluated at current time $t_n$) leaving the unknown linear terms at $t_{n+1}$. For example, the semi-implicit Crank-Nicolson method is
\begin{equation}\label{eqn:CN_SI}
{\bf U}^{n+1} = \left(I - \frac{\dt}{2}D\right)^{-1} \left({\bf U}^n + \frac{\dt}{2}D {\bf U}^n + \dt {\bf R}({\bf U}^n) +{\bf B}\right).
\end{equation}

\begin{exercise}
\renewcommand{\labelenumi}{\alph{enumi})}
After reviewing Example \ref{ex:FWD} and Example \ref{ex:BWD}, complete the following for the semi-implicit Crank-Nicolson method (\ref{eqn:CN_SI}).
\begin{enumerate}
\item Determine the two-level matrix form for the Test equation (\ref{eqn:LinearTestIBVP}). Note, set ${\bf R = 0}$.
\item *Determine the two-level matrix form for the Fisher-KPP equation (\ref{eqn:FisherIBVP}).
\end{enumerate}
\renewcommand{\labelenumi}{\arabic{enumi}.}
*See section \ref{sub:stability} for the answer and its analysis.
\end{exercise}

Taylor series expansions can be used to prove that while the Crank-Nicolson method (\ref{eqn:CN}) for the linear Test equation (\ref{eqn:LinearTestIBVP}) is second order accurate (See Definition \ref{def:order-accuracy} in section \ref{App:AccuracyProof}), the semi-implicit Crank-Nicolson method (\ref{eqn:CN_SI}) for a nonlinear PDE is only first order accurate in time. To increase truncation error accuracy, unknown terms in an implicit method can be truncated using a more accurate explicit method. For example, blending the Crank-Nicolson method with Forward Euler approximations creates the Improved Euler Crank-Nicolson method, which is second order accurate in time for nonlinear PDEs. 
\begin{eqnarray}\label{eqn:CN_IE}
{\bf U}^* &=& {\bf U}^n + \dt\left(D {\bf U}^n + {\bf R}({\bf U}^n)+{\bf B}\right),\\\nonumber
{\bf U}^{n+1} &=& {\bf U}^n + \frac{\dt}{2} D\left( {\bf U}^{n} +  {\bf U}^{*}\right) + \frac{\dt}{2}\left({\bf R}({\bf U}^{n}) + {\bf R}({\bf U}^{*})\right)+\dt{\bf B}.
\end{eqnarray}
This improved Euler Crank-Nicolson method (\ref{eqn:CN_IE}) is part of the family of Runge-Kutta methods which embed a sequence of truncated Taylor expansions for implicit terms to create an explicit method of any given order of accuracy \cite{Burden2011}. 
Proofs of the accuracy for the semi-implicit (\ref{eqn:CN_SI}) and improved Euler (\ref{eqn:CN_IE}) methods are included in Appendix \ref{App:AccuracyProof}.

\begin{exercise}
\renewcommand{\labelenumi}{\alph{enumi})}
After reviewing Example \ref{ex:FWD} and Example \ref{ex:BWD}, complete the following for the Improved Euler Crank-Nicolson method (\ref{eqn:CN_IE}).
\begin{enumerate}
\item Determine the two-level matrix form for the Test equation (\ref{eqn:LinearTestIBVP}). Note, set ${\bf R = 0}$.
\item Determine the two-level matrix form for the Fisher-KPP equation (\ref{eqn:FisherIBVP}).
\end{enumerate}
\renewcommand{\labelenumi}{\arabic{enumi}.}
\end{exercise}

\subsection{Overview of Software}
\label{sub:software}
Several software have been developed to compute numerical methods. Commercially, MATLAB, Mathematica, and Maple are the best for analyzing such methods, though there are other commercial software like COMSOL which can do numerical simulation with much less work on your part. Open-source software capable of the same (or similar) numerical computations, such as Octave, SciLab, FreeFEM, etc. are also available. Once the analysis is complete and methods are fully tested, simulation algorithms are trimmed down and often translated into Fortran or C/C++ for efficiency in reducing compiler time.

We will focus on programming in MATLAB, created by mathematician Cleve Moler (born in 1939), one of the authors of the LINPACK and EISPACK scientific subroutine libraries used in Fortran and C/C++ compilers \cite{Moler2008}. Cleve Moler originally created MATLAB to give his students easy access to these subroutines without having to write in Fortran or C themselves. In the same spirit, we will be working with simple demonstration programs, listed in the appendix, to access the core ideas needed for our numerical investigations. Programs {\tt PDE\_Solution.m} (Appendix \ref{App:Solutions}), {\tt PDE\_Analysis\_Setup.m} (Appendix \ref{App:Setup}), and  {\tt Method\_Accuracy\_Verification.m} (Appendix \ref{App:Accuracy}) are MATLAB scripts, which means they can be run without any direct input and leave all computed variables publicly available to analyze after they are run. Programs {\tt CrankNicolson\_SI.m} (Appendix \ref{App:CN}) and {\tt Newton\_System.m} (Appendix \ref{App:Newton}) are MATLAB functions, which means they may require inputs to run, keep all their computations private, and can be effectively embedded in other functions or scripts. All demonstrations programs are run through {\tt PDE\_Solution.m}, which is the main program for this group.

\begin{example}\label{ex:demo}
The demonstration programs can be either downloaded from the publisher or typed into five separate MATLAB files and saved according to the name at the top of the file (e.g. {\tt PDE\_Analysis\_Setup.m}). To run them, open MATLAB to the folder which contains these five programs. In the command window, type {\tt help PDE\_Solution} to view the comments in the header of the main program. Then type {\tt PDE\_Solution} to run the default demonstration. This will solve and analyze the Fisher-KPP equation (\ref{eqn:FisherIBVP}) using the default parameters, produce five graph windows, and report three outputs on the command window. The first graph is the numerical solution using MATLAB's built-in implementation of the Rosenbrock method ({\tt ode23s}), which is also demonstrated in Figure \ref{fig:lf5ode23s} (right). The second graph plots the comparable eigenvalues for the semi-implicit Crank-Nicolson method (\ref{eqn:CN_SI}) based upon the maximum $\dt$ step used in the chosen method (Rosenbrock by default).  The third graph shows a different steady-state solution to the Fisher-KPP equation (\ref{eqn:FisherIBVP}) found using Newton's method (\ref{eqn:Newton}). The fourth graph shows the rapid reduction of the error of this method as the Newton iterations converge. The fifth graph shows the instability of the Newton steady-state solution by feeding a noisy perturbation of it back into the Fisher-KPP equation (\ref{eqn:FisherIBVP}) as an initial condition. This noisy perturbation is compared to round-off perturbation in Figure \ref{fig:lf5convergence} to see how long this Newton steady-state solution can endure. Notice that the solution converges back to the original steady-state solution found in the first graph.
\end{example}
To use the semi-implicit Crank-Nicolson method (\ref{eqn:CN_SI}) instead of MATLAB's ode23s, do the following. Now the actual eigenvalues of this method are plotted in the second graph. 
\begin{exercise}
Open {\tt PDE\_Solution.m} in the MATLAB Editor. Then comment lines 7-9 (type a \% in front of each line) and uncomment lines 12-15 (remove the \% in front of each line). Run {\tt PDE\_Solution}. Verify that the second graph matches Figure \ref{fig:f5eig}.
\end{exercise}
The encoded semi-implicit Crank-Nicolson method (\ref{eqn:CN_SI}) uses a fixed step size $\dt$, so it is generally not as stable as MATLAB's built-in solver. It would be best to now uncomment lines 7-9 and comment lines 12-15 to return to the default form before proceeding.  The main benefit of the ode23s solver is that it is adaptive in choosing the optimal $\dt$ step size and adjusting it for regions where the equation is easier or harder to solve than before. This method is also sensitive to {\it stiff} problems, where stability conditions are complicated or varying. MATLAB has several other built-in solvers to handle various situations. You can explore these by typing {\tt help ode}.

Once you have run the default settings, open up {\tt PDE\_Analysis\_Setup} in the editor and tweak the equation parameter values {\tt a,b,L,delta,rho,degree,c} and logistical parameter {\tt dx}. After each tweak, make sure you run the main program {\tt PDE\_Solution}. The logistical parameters {\tt tspan,dt} for the numerical method can also be tweaked in {\tt PDE\_Solution}, and an inside view of Newton iterations can be seen by uncommenting lines 38-39. Newton's method is covered in Section \ref{sub:Newtonmethod}. A solution with Newton's method is demonstrated in \ref{fig:f1steadystate}(left), while all of the iterations are graphed in Figure \ref{fig:f1steadystate}(right). Note that Figure \ref{fig:f1steadystate}(right) is very similar to a solution which varies over time, but it is not. The graph of the iterations demonstrates how Newton's method seeks better and better estimates of a fixed steady-state solution discussed in Section \ref{sub:steadystate}.
\begin{exercise}
In {\tt PDE\_Analysis\_Setup}, set parameter values, $a=0, b=0, L=10, \delta = \frac{1}{10}, \dx=\frac{1}{20}, \rho=1, {\rm degree}=2, c=1$. Then, in {\tt PDE\_Solution}, uncomment lines 38-39 and run it. Verify that the third and fourth graphs matches Figure \ref{fig:f1steadystate}.
\end{exercise}
Notice that the iterations of Newton's method in Figure \ref{fig:f1steadystate}(right) demonstrate oscillatory behavior in the form of waves which diminish in amplitude towards the steady-state solution. These are referred to as stable numerical oscillations similar to the behavior of an underdamped spring \cite{Burden2011}. These stable oscillations suggest that the steady-state solution is stable (attracting other initial profiles to it), but due to the negative values in the solution in Figure \ref{fig:f1steadystate}(left), it is actually an unstable steady-state for Fisher-KPP equation (\ref{eqn:FisherIBVP}). You can see this demonstrated in the last graph plotted when you ran  {\tt PDE\_Solution} where there is a spike up to a value around $-4 \times 10^12$. This paradox demonstrates that not all steady-state solutions are stable and that the stability of Newton's method differs from the stability of a steady-state solution to an IBVP.

\begin{figure}
\centerline{
\includegraphics[scale=.35]{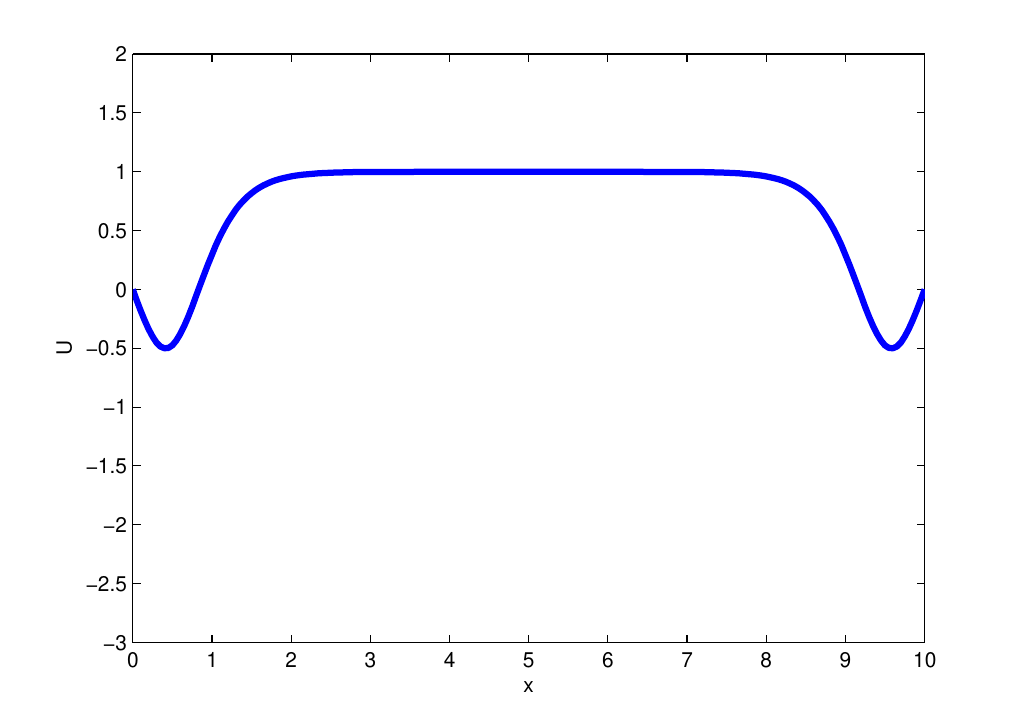} 
\includegraphics[scale=.45]{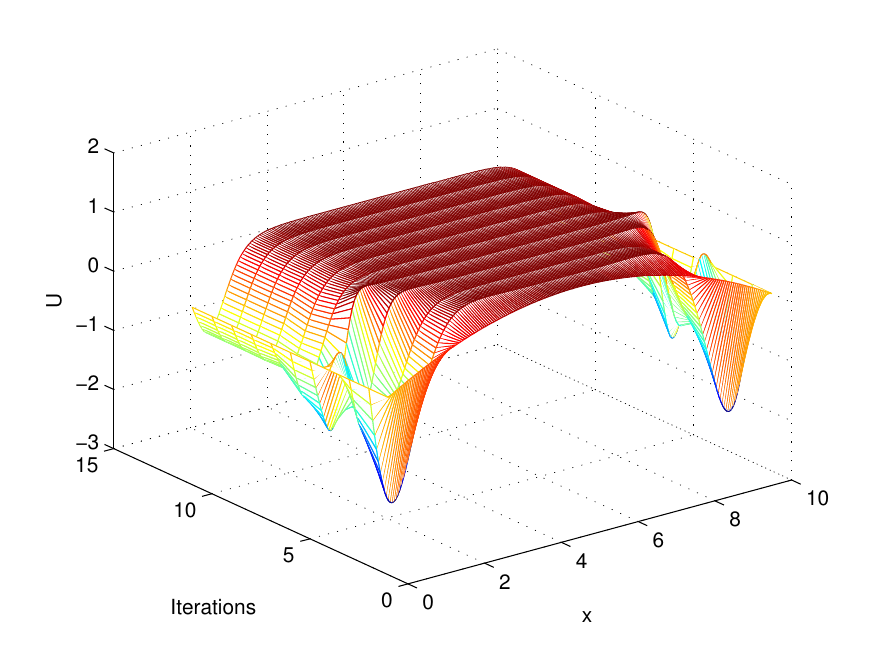} 
}
\caption{An example steady-state solution using Newton's method (left) and the iterations to that steady-state (right) using the parameter values $a=0, b=0, L=10, \delta = \frac{1}{10}, \dx=\frac{1}{20}, \rho=1, {\rm degree}=2, c=1$ and initial condition from line 8 in {\tt PDE\_Analysis\_Setup.m} found in Appendix \ref{App:Setup}. Also, uncomment lines 38-39 in {\tt PDE\_Solution.m} found in Appendix \ref{App:Solutions}.} 
\label{fig:f1steadystate}
\end{figure}

Some best practices of programming in MATLAB are to clean up before running new computations, preallocate memory, and store calculations which are used more than once. Before new computations are stored in a script file, you can clean up your view of previous results in the command window ({\tt clc}), delete previously held values and structure of all ({\tt clear all}) or selected ({\tt clear {\it name1, name2}}) variables, and close all ({\tt close all}) or selected ({\tt close handle1, handle2}) figures. The workspace for running the main program {\tt PDE\_Solution} is actually cleared in line 5 of {\tt PDE\_Analysis\_Setup} so that this supporting file can be run independently when needed. 

When you notice the same calculation being computed more than once in your code, store it as a new variable to trim down the number of calculations done for increased efficiency. Most importantly, preallocate the structure of a vector or matrix that you will fill in values with initial zeros ({\tt zeros(columns,rows)}), so that MATLAB does not create multiple copies of the variable in memory as you fill it in. The code {\tt BCs = zeros(M-2,1);} in line 23 of {\tt PDE\_Analysis\_Setup.m} is an example of preallocation for a vector. Preallocation is one of most effective ways of speeding up slow code.
\begin{exercise}\label{prob:preallocate}
Find the four (total) lines of code in {\tt Newton\_System.m}, {\tt Method\_ Accuracy\_Verification.m}, and {\tt CrankNicolson\_SI.m} which preallocate a variable.
\end{exercise}

\section{Error Analysis}
\label{sec:analysis}
To encourage confidence in the numerical solution it is important to support the theoretical results with numerical demonstrations. For example, a theoretical condition for stability or oscilation-free behavior can be demonstrated by comparing solutions before and after the condition within a small neighborhood of it. On the other hand, the order of accuracy can be demonstrated by comparing subsequent solutions over a sequence of step sizes, as we will see in section \ref{sub:verifyerror}. Demonstrating stability ensures {\it when}, while accuracy ensures {\it how rapidly}, the approximation will converge to the true solution. Showing when oscillations begin to occur prevents any confusion over the physical dynamics being simulated, as we will investigate in section \ref{sub:oscillations}.

\subsection{Verifying Accuracy}
\label{sub:verifyerror}
Since numerical methods for PDEs use arithmetic to approximate the underlying calculus, we expect some error in our results, including inaccuracy measuring the distance from our target solution as well as some imprecision in the variation of our approximations. We must also balance the mathematical accuracy in setting up the method with the round-off errors caused by computer arithmetic and storage of real numbers in a finite representation. As we use these values in further computations, we must have some assurance that the error is minimal. Thus, we need criteria to describe how confident we are in these results. 

Error is defined as the difference between the true value $u(x_m,t_n)$ and approximate value $U_m^n$, but this value lacks the context given by the magnitude of the solution's value and focuses on the error of individual components of a solution's vector.
Thus, the relative error $\epsilon$ is more meaningful as it presents the absolute error relative to the true value as long as $u(x_m,t_n)\neq 0$ under a suitable norm such as the max norm $||\cdot||_\infty$.
\begin{definition}\label{def:rel-err}
The {\it relative error} $\epsilon$ for a vector solution ${\bf U}^n$ is the difference between the true value $u(x_m,t_n)$ and approximate value $U_m^n$ under a suitable norm $||\cdot||$, relative to the norm of the true value as
\begin{equation}
\epsilon =\frac{||{\bf u}({\bf x},t_n) - {\bf U}^n||}{||{\bf u}({\bf x},t_n)||}\times 100\%.
\end{equation}
\end{definition}

The significant figures of a computation are those that can be claimed with confidence. They correspond to a number of confident digits plus one estimated digit, conventionally set to half of the smallest scale division on the measurement device, and specified precisely in Definition \ref{def:sig-figs}.
\begin{definition}\label{def:sig-figs}
The value $U_m^n$ approximates $u(x_m,t_n)$ to $N$ {\it significant digits} if $N$ is the largest non-negative integer for which the relative error is bounded by the significant error $\epsilon_s(N)$
\begin{equation}
\epsilon_s(N) = \left(5 \times 10^{-N}\right)\times100\%
\end{equation}
\end{definition}
To ensure all computed values in an approximation have about $N$ significant figures, definition \ref{def:sig-figs} implies 
\begin{equation}\label{eqn:sig-bound}
N+1>\log_{10}\left(\frac{5}{\epsilon}\right)>N,
\end{equation}

Although the true value is not often known, the relative error of a previous approximation can be estimated using the best available approximation in place of the true value. 

\begin{definition}\label{def:eqn:app-rel-err}
For an iterative method with improved approximations ${\bf U}^{(0)},{\bf U}^{(1)},$ ..., ${\bf U}^{(k)}$, ${\bf U}^{(k)}$, the {\it approximate relative error} at position $(x_m,t_n)$ is defined \cite{Burden2011} as the difference between current and previous approximations relative to the current approximation, each under a suitable norm $||\cdot||$
\begin{eqnarray}\label{eqn:app-rel-err}
E^{(k)} 
&=& \frac{||{\bf U}^{(k+1)}-{\bf U}^{(k)}||}{||{\bf U}^{(k+1)}||}\times 100\% 
\end{eqnarray}
closely approximates, $\epsilon^{(k)}$, the relative error for the $k^{\rm th}$ iteration assuming that the iterations are converging (that is, as long as  $\epsilon^{(k+1)}$ is much less than $\epsilon^{(k)}$). 
\end{definition}

The following conservative theorem, proven in \cite{Scarborough1966}, is helpful in clearly presenting the lower bound on the number of significant figures of our results.
\begin{theorem}\label{thm:significant-figures}
Approximation ${\bf U}^n$ at step $n$ with approximate relative error ${E}^{(k)}$ is correct to at least $N-1$ significant figures if 
\begin{equation}\label{eqn:sig-crit}
{E}^{(k)}<\epsilon_s(N)
\end{equation}

\end{theorem}
Theorem \ref{thm:significant-figures} is conservatively true for the relative error $\epsilon$, often underestimating the number of significant figures found. The approximate relative error ${ E}^{(k)}$ (\ref{eqn:app-rel-err}), however, underestimates the relative error $\epsilon$ and may predict one significant digit more for low order methods. 

Combining theorem \ref{thm:significant-figures} with equation (\ref{eqn:sig-bound}), the number of significant figures has a lower bound 
\begin{equation}\label{eqn:sig-approx-bound} 
N \geq \left\lfloor \log_{10}\left(\frac{0.5}{{ E}^{(k)}}\right)\right\rfloor.
\end{equation}

\begin{table}
\caption{Verifying Accuracy in Time for Semi-Implicit Crank-Nicolson Method}
\label{tab:CN}      

\begin{tabular}{p{1cm}p{2.cm}p{1cm}p{1.5cm}p{2cm}p{1cm}p{1.5cm}}
\hline\noalign{\smallskip}
 & \multicolumn{2}{c}{\bf Test Equation} & \multicolumn{2}{c}{\bf Fisher-KPP Equation} \\
$\Delta t$ & Approximate Error$^a$ & Sig. Figs$^c$& \centering Order~of Accuracy$^b$ & Approximate Error$^a$ &  Sig. Figs$^c$& \centering Order~of Accuracy$^b$ \tabularnewline
\noalign{\smallskip}\svhline\noalign{\smallskip}
1		 &  3.8300e-05 &4 & 2 (2.0051)  & 4.1353e-05 &4& -1 (-0.5149)\\
$\frac12$ & 9.5413e-06 &4& 2 (2.0009)  & 5.9091e-05 &3& 0 (0.4989)\\
$\frac14$ & 2.3838e-06 &5& 2 (2.0003)  & 4.1818e-05 &4& 1 (0.7773)\\
$\frac18$ & 5.9584e-07 &5& 2 (2.0001)  & 2.4399e-05 &4 & 1 (0.8950)\\
$\frac1{16}$ & 1.4895e-07&6& 2 (2.0000)   & 1.3120e-05 &4 & 1 (0.9490)\\
$\frac1{32}$ & 3.7238e-08&7& 2 (2.0000)  & 6.7962e-06 &4& 1 (0.9749)\\
$\frac1{64}$ & 9.3096e-09&7& 2 (2.0001)  & 3.4578e-06 &5& 1 (0.9875)\\
$\frac1{128}$ & 2.3273e-09&8& 2 (1.9999)  &  1.7439e-06 &5 & 1 (0.9938)\\
\noalign{\smallskip}\hline\noalign{\smallskip}
\end{tabular}\\
$^a$ Approximate error under the max norm $||\cdot||_\infty$ for numerical solution ${\bf U}^{(k)}$ computed at $t_n=10$ compared to solution at next iteration ${\bf U}^{(k+1)}$ whose time step is cut in half.\\
$^b$ Order of accuracy is measured as the power of 2 dividing the error as the step size is divided by 2\\
$^c$ Minimum number of significant figures predicted by approximate error bounded by the significant error $\epsilon_s(N)$ as in equation (\ref{eqn:sig-crit})


\end{table}

\begin{example}
Table \ref{tab:CN} presents these measures of error to analyze the Crank-Nicolson method (\ref{eqn:CN}) for the linear Test equation (\ref{eqn:LinearTestIBVP}) and the semi-implicit version of the Crank-Nicolson method (\ref{eqn:CN_SI}) for the Fisher-KPP equation (\ref{eqn:FisherIBVP}). Column 1 tracks the step size $\dt$ as it is halved for improved approximations ${\bf U}^{(k)}$ at $t=10$. Columns 2 and 5 present the approximate errors in scientific notation for easy readability. Scientific notation helps read off the minimum number of significant figures ensured by Theorem \ref{thm:significant-figures}, as presented in columns 3 and 6. Notice how the errors in column 2 are divided by about 4 each iteration while those in column 5 are essentially divided by 2. This ratio of approximate relative errors demonstrates the orders of accuracy, $p$ as a power of 2 since the step sizes are divided by 2 each iteration.
\begin{equation}
\frac{\epsilon^{(k+1)}}{\epsilon^{(k)}} = \frac{C}{2^p},
\end{equation}
for some positive scalar $C$ which underestimates the integer $p$ when $C>1$ and overestimates $p$ when $C<1$. By rounding to mask the magnitude of $C$, the order $p$ can be computed as 
\begin{equation}\label{eqn:order-accuracy}
p = {\rm round}\left(\log_2\left(\frac{\epsilon^{(k)}}{\epsilon^{(k+1)}}\right)\right).
\end{equation}
Columns 4 and 7 present both rounded and unrounded measures of the order of accuracy for each method and problem. Thus, we have verified that Crank-Nicolson method (\ref{eqn:CN}) on a linear problem is second order accurate in time, whereas the semi-implicit version of the Crank-Nicolson method (\ref{eqn:CN_SI}) for the nonlinear Fisher-KPP equation (\ref{eqn:FisherIBVP}) is only first order in time.

For comparison, Table \ref{tab:ode23s} presents these same measures for the Rosenbrock method built into MATLAB as ode23s. See example program {\tt Method\_Accuracy \_Verification.m} in Appendix \ref{App:AccuracyProof} for how to fix a constant step size in such an adaptive solver by setting the initial step and max step to be $\dt$ with a high tolerance to keep the adaptive method from altering the step size.
\end{example}

\begin{exercise}
Implement the Improved Euler Crank-Nicolson method (\ref{eqn:CN_IE}) and verify that the error in time is $O\left(\dt^2\right)$ on both the Test equation (\ref{eqn:LinearTestIBVP}) and Fisher-KPP equation (\ref{eqn:FisherIBVP}) using a table similar to Table \ref{tab:CN}.
\end{exercise}


\begin{table}
\caption{Verifying Accuracy in Time for ode23s Solver in MATLAB}
\label{tab:ode23s}      

\begin{tabular}{p{1cm}p{2.cm}p{1cm}p{2cm}p{2cm}p{1cm}p{2cm}}
\hline\noalign{\smallskip}
 & \multicolumn{2}{c}{\bf Test Equation} & \multicolumn{2}{c}{\bf Fisher-KPP Equation} \\
\centering $\Delta t$ & \centering Approximate Error$^a$ & \centering Sig. Figs$^c$& \centering Order~of Accuracy$^b$ & \centering Approximate Error$^a$ & \centering  Sig. Figs$^c$& \centering Order~of Accuracy$^b$ \tabularnewline
1 &1.8829e-05&4&   2 (2.00863)&	   7.4556e-05&3&	 2 (1.90005)\\
$\frac12$  &4.6791e-06&5&	 2 (2.00402)&	   1.9976e-05&4&	 2 (1.98028)\\	 
$\frac14$  &1.1665e-06&5&	 2 (2.00198)&	   5.0628e-06&4&	 2 (1.99722)\\	 
$\frac18$  &2.9123e-07&6&	 2 (2.00074)&	   1.2681e-06&5&	 2 (2.00060)\\	 
$\frac1{16}$ &7.2771e-08&6&	 2 (2.00054)&	   3.169e-07& 6&  2 (2.00025)\\	 
$\frac1{32}$&1.8186e-08	&7&  2 (2.00027)&	   7.9211e-08&6&	 2 (1.99900)\\	 
$\frac1{64}$ &4.5456e-09&8&	 2 (1.99801)&	   1.9817e-08&7&	 2 (2.00168)\\	 
$\frac1{128}$&1.138e-09& 8&   2 (1.99745)&	   4.9484e-09&8&	 2 (2.00011)\\	
\noalign{\smallskip}\hline\noalign{\smallskip}
\end{tabular}\\
$^a$ Approximate error under the max norm $||\cdot||_\infty$ for numerical solution ${\bf U}^{(k)}$ computed at $t_n=10$ compared to solution at next iteration ${\bf U}^{(k+1)}$ whose time step is cut in half.\\
$^b$ Order of accuracy is measured as the power of 2 dividing the error as the step size is divided by 2\\
$^c$ Minimum number of significant figures predicted by approximate error bounded by the significant error $\epsilon_s(N)$ as in equation (\ref{eqn:sig-crit})



\end{table}

\subsection{Convergence}\label{sub:convergence}
Numerical methods provide dependable approximations $U_{m}^{n}$ of
the exact solution $u\left(x_{m},t_{n}\right)$ only if the approximations
converge to the exact solution, $U_{m}^{n}\to u\left(x_{m},t_{n}\right)$
as the step sizes diminish, $\dx,\dt\to0$. Convergence of a numerical
method relies on both the consistency of the approximate equation
to the original equation as well as the stability of the solution
constructed by the algorithm. Since consistency is determined by construction,
we need only analyze the stability of consistently constructed schemes
to determine their convergence. This convergence through stability is proven generally by the Lax-Richtmeyer Theorem \cite{Thomas1995}, but is more specifically defined for
two-level numerical methods (\ref{eqn:two-level_method}) in the Lax Equivalence Theorem (\ref{thm:Lax-Equivalence-Theorem}).

\begin{definition}\label{def:well-posed}
A problem is {\it well-posed} if there exists a unique solution which depends continuously on the conditions.
\end{definition}
Discretizing an initial-boundary-value problem (IBVP)  into an initial-value problem (IVP) as an ODE system ensures the boundary conditions are well developed for the problem, but the initial conditions must also agree at the boundary for the problem to be well-posed. Further, the slope function of the ODE system needs to be infinitely differential, like the Fisher-KPP equation (\ref{eqn:FisherIBVP}), or at least Lipshitz-continuous, so that Picard's uniqueness and existence theorem via Picard iterations \cite{Burden2011} applies to ensure that the problem is well-posed \cite{Thomas1995}. Theorem \ref{thm:Lax-Equivalence-Theorem}, proved in \cite{Thomas1995} ties this altogether to ensure convergence of the numerical solution to the true solution.

\begin{theorem}[Lax Equivalence Theorem] A consistent, two-level difference scheme
(\ref{eqn:two-level_method}) for a well-posed linear IVP is convergent if and only if it is stable. \label{thm:Lax-Equivalence-Theorem}
\end{theorem}

\subsection{Stability}\label{sub:stability} 
The beauty of Theorem \ref{thm:Lax-Equivalence-Theorem} (Lax Equivalence Theorem ) is that once we have a consistent numerical method, we can explore the bounds on stability to ensure convergence of the numerical solution. We begin with a few definitions and examples to lead us to von Neumann stability analysis, named after mathematician John von Neumann (1903-1957).

Taken from the word {\it eigenwerte}, meaning one's own values in German, the {\it eigenvalues} of a matrix define how a matrix operates in a given situation. 
\begin{definition}\label{def:eig}
For a $k\times k$ matrix $M$, a scalar $\lambda$ is an eigenvalue of $M$ with corresponding $k\times 1$ eigenvector ${\bf v}\neq {\bf 0}$ if 
\begin{equation}
M{\bf v} = \lambda {\bf v}.
\end{equation}
\end{definition}

\begin{figure}[t]
\centerline{
\includegraphics[scale=.65]{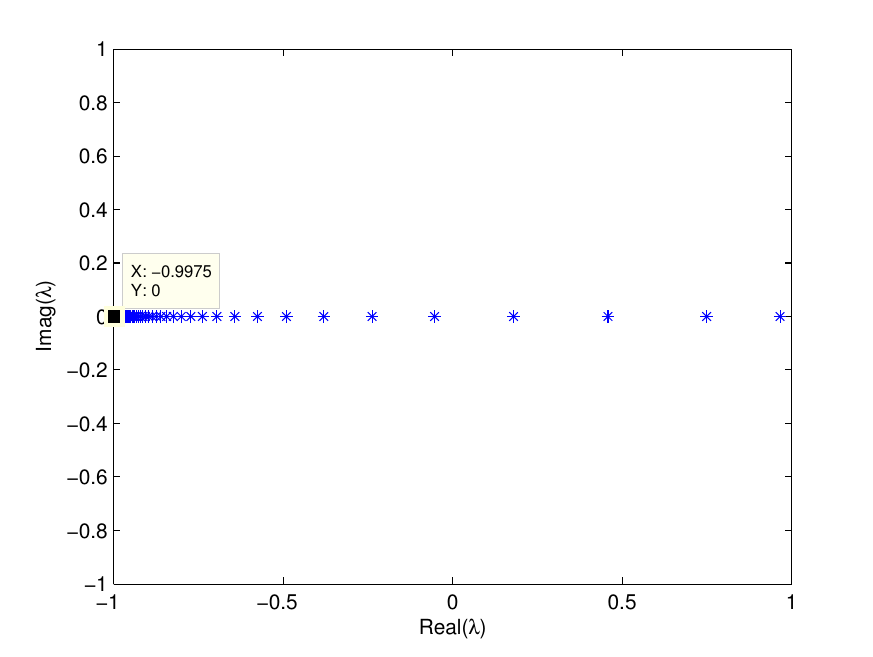} 
}
\caption{Plot of real and imaginary components of all eigenvalues of method matrix $M$ for semi-implicit Crank-Nicolson method for Fisher-KPP equation (\ref{eqn:FisherIBVP}) using the default parameter values $a=0, b=1, L=10, \delta=1, \dx=0.05, \rho=1, {\rm degree}=2, c=\frac{1}{3}$ and initial condition as given in line 16 of {\tt PDE\_Analysis\_Setup.m} in Appendix \ref{App:Setup}. }
\label{fig:f5eig}
\end{figure}

Lines 22-31 of the demonstration code {\tt PDE\_Solution.m}, found in Appendix \ref{App:Solutions}, compute several measures helpful in assessing stability, including the graph of the eigenvalues of the method matrix for a two-level method (\ref{eqn:two-level_method}) on the complex plane. The spectrum, range of eigenvalues, of the default method matrix is demonstrated in Figure \ref{fig:f5eig}, while the code also reports the range in the step size ratio $\frac{}{}$, range in the real parts of the eigenvalues, and computes the spectral radius.
\begin{definition}
The spectral radius of a matrix, $\mu(M)$ is the maximum magnitude of all eigenvalues of $M$
\begin{equation}
\mu(M)=\max_{i}|\lambda_{i}|.
\end{equation} 
\end{definition}

The norm of a vector is a well-defined measure of its size in terms of a specified metric, of which the Euclidean distance (notated $||\cdot||_2$), the maximum absolute value ($||\cdot||_\infty$), and the absolute sum ($||\cdot||_1$) are the most popular. See \cite{HornJohnson2012} for further details. These measures of a vector's size can be extended to matrices. 
\begin{definition}\label{defn:MatrixNorm}
For any norm $||\cdot||$, the corresponding matrix norm $|||\cdot|||$ is defined by
\begin{equation}
|||M|||=\max_{\bf x}\frac{||M{\bf x}}{\bf x}.
\end{equation}
\end{definition}
A useful connection between norms and eigenvalues is the following theorem \cite{HornJohnson2012}
\begin{theorem} \label{thm:RadiusBound}
For any matrix norm $|||\cdot|||$ and square matrix $M$, $\mu(M)\leq |||M|||$.
\end{theorem}
\begin{proof}
Consider an eigenvalue $\lambda$ of matrix $M$ corresponding to eigenvector ${\bf x}$ whose magnitude equals the spectral radius, $|\lambda|=\mu(M)$. Form a square matrix $X$ whose columns each equal the eigenvector ${\bf x}$. Note that by Definition \ref{def:eig}, $MX=\lambda X$ and $|||X|||\neq 0$ since ${\bf x}\neq{\bf 0}$.
\begin{eqnarray}
|\lambda | |||X||| &=& ||\lambda X||\\ \nonumber
&=& ||M X||\\ \nonumber
&\leq & |||M||| |||X|||
\end{eqnarray}
Therefore, $|\lambda|=\mu(M)\leq |||M|||$.
\end{proof}
Theorem \ref{thm:RadiusBound} can be extended to an equality in Theorem \ref{thm:RadiusNorm} (proven in \cite{HornJohnson2012}),
\begin{theorem}\label{thm:RadiusNorm}
$\mu(M)= \lim_{k\to\infty} |||M^k|||^\frac{1}{k}$.
\end{theorem}
This offers a useful estimate of the matrix norm by the spectral radius, specifically when the matrix is powered up in solving a numerical method.

Now, we can apply these definitions to the stability of a numerical method. An algorithm is stable if small changes in the initial data produce only small changes in the final results \cite{Burden2011}, that is, the errors do not ``grow too fast" as quantified in Definition \ref{def:stable}.
\begin{definition}\label{def:stable}
A two-level difference method (\ref{eqn:two-level_method}) is said to be stable with respect to the norm $||\cdot||$ if there exist positive max step sizes $\dt_{0},\ and\ \dx_{0},$ and non-negative constants $K$ and $\beta$ so that 
\[
||{\bf U}^{n+1}||\leq Ke^{\beta\triangle t}||U^{0}||,
\]
for $0\leq t,\ 0<\triangle x\leq\triangle x_{0}\ and\ 0<\triangle t\leq\triangle t_{0}$.
\label{defn:Burden Faires Stability}
\end{definition}

The {\it von Neumann criterion} for stability (\ref{VN criterion}) allows for stable solution to an exact solution which is not growing (using $C=0$ for the tight von Neumann criterion) or at most exponentially growing (using some $C>0$) by bounding the spectral radius of the method matrix,
\begin{equation}
\mu(M)\leq1+C\triangle t,\label{VN criterion}
\end{equation}
for some $C\geq0$.

Using the properties of norms and an estimation using Theorem \ref{thm:RadiusNorm}, we can approximately bound the size of the numerical solution under the von Neumann criterion as
\begin{eqnarray}
||{\bf U}^{n+1}|| &=& ||M^{n+1}U^{0}||\\ \nonumber
&\leq & |||M^{n+1}|||\ ||U^{0}||\\ \nonumber
&\approx & \mu(M)^{n+1}\ ||U^{0}||\\ \nonumber
&\leq & \left(1+C\dt \right)^{n+1}\ ||U^{0}||\\ \nonumber
&=& \left(1+(n+1)C\dt +...\right)\ ||U^{0}||\\ \nonumber
&\leq & e^{(n+1)C\dt}\ ||U^{0}||\\ \nonumber
&=& Ke^{\beta\dt}\ ||U^{0}||
\end{eqnarray}
for $K=1,\beta=(n+1)C$, which makes the von Neumann criterion sufficient for stability of the solution in approximation for a general method matrix. When the method matrix is symmetric, which occurs for many discretized PDEs including the Test equation with Forward Euler, Backward Euler, and Crank-Nicolson methods, the spectral radius equals the $|||\cdot|||_2$ of the matrix. Then, the von Neumann criterion (\ref{VN criterion}) provides a precise necessary and sufficient condition for stability \cite{Thomas1995}. 


If the eigenvalues are easily calculated, they provide a simple means for predicting the stability and behavior of the solution. {\it Von Neumann analysis}, estimation of the eigenvalues from the PDE itself, provides a way to extract information about the eigenvalues, if not the exact eigenvalues themselves.


For a two-level numerical scheme (\ref{eqn:two-level_method}), the
eigenvalues of the combined transformation matrix indicate the stability
of the solution. 
Seeking a solution to the linear difference scheme by separation of variables, as is used for linear PDEs, we can show that the discrete error growth factors
are the eigenvalues of the method matrix $M$.
Consider a two-level difference scheme (\ref{eqn:two-level_method}) for a linear parabolic PDE so that ${\bf R}=0$, 
then the eigenvalues can be defined by the constant ratio \cite{Thomas1995} 
\begin{equation}
\frac{U_{m}^{n+1}}{U_{m}^{n}}=\frac{T^{n+1}}{T^{n}}=\lambda_{m},\ \mbox{where }U_{m}^{n}=X_{m}T^{n}.\label{Growth Factor-Eigenvalue relation}
\end{equation}
The error $\epsilon_{m}^{n}=u\left(x_{m},t_{n}\right)-U_{m}^{n}$
satisfies the same equation as the approximate solution $U_{m}^{n}$,
so the eigenvalues also define the ratio of errors in time
called the error growth (or amplification) factor \cite{Thomas1995}.
Further, the error can be represented in Fourier form as $\epsilon_{m}^{n}=\hat{\epsilon}e^{\alpha n\dt}e^{im\beta\dx}$
where $\hat{\mbox{\ensuremath{\epsilon}}}$ is a Fourier coefficient,
$\alpha$ is the growth/decay constant, $i=\sqrt{-1}$, and $\beta$ is the wave
number. Under this assumptions, the eigenvalues are equivalent to the error growth factors of the numerical method,
\[
\lambda_{k}=\frac{U_{m}^{n+1}}{U_{m}^{n}}=\frac{\epsilon_{m}^{n+1}}{\epsilon_{m}^{n}}=\frac{\hat{\epsilon}e^{\alpha\left(n+1\right)\dt}e^{im\beta\dx}}{\hat{\epsilon}e^{\alpha n\dt}e^{im\beta\dx}}=e^{\alpha\dt}.
\]
We can use this equivalency to finding bounds on 
the eigenvalues of a numerical scheme by plugging the representative growth factor $e^{\alpha\dt}$ into the method discretization called von Neumann stability
analysis (also known as Fourier stability analysis) \cite{vonNeumann1950, Thomas1995}.
The von Neumann criterion (\ref{VN criterion}) ensures that the matrix stays bounded as it is powered up. 
If possible, $C$ is set to 0, called the tight von Neumann criterion for simplified bounds on the step sizes.
As another consequence of this equivalence, analyzing the spectrum of the transformation matrix also reveals patterns in the orientation, spread, and balance of the growth of errors for various wave modes.

Before we dive into the stability analysis, it is helpful to review some identities for reducing the error growth factors:
\begin{eqnarray}\label{eqn:identities}
\frac{e^{ix}+ e^{-ix}}{2} = \cos(x),\ \frac{e^{ix}- e^{-ix}}{2} =i\sin(x),\\ \nonumber
\frac{1- \cos(x)}{2} = \sin^2\left(\frac{x}{2}\right),\frac{1+ \cos(x)}{2} =\cos^2\left(\frac{x}{2}\right).
\end{eqnarray}

\begin{example}\label{ex:FWD_Stability}
Use the von Neumann stability analysis to determine conditions on $\dt,\dx$ to ensure stability of the Crank-Nicolson method (\ref{eqn:CN}) for the Test equation (\ref{eqn:LinearTestIBVP}).

For von Neumann stability analysis, we replace each solution term $U_m^n$ in the method with the the representative error $\epsilon_{m}^{n}=\hat{\epsilon}e^{\alpha n\dt}e^{im\beta\dx}$,
\begin{eqnarray}
U_{m}^{n+1}&=& rU_{m-1}^{n}+\left(1-2r\right)U_{m}^{n} +rU_{m+1}^{n},\\ \nonumber
\hat{\epsilon}e^{\alpha n\dt+\alpha\dt}e^{im\beta\dx} &=& r\hat{\epsilon}e^{\alpha n\dt}e^{im\beta\dx-i\beta\dx} + \left(1-2r\right)\hat{\epsilon}e^{\alpha n\dt}e^{im\beta\dx}\\ \nonumber
&&+ r\hat{\epsilon}e^{\alpha n\dt}e^{im\beta\dx+i\beta\dx},
\end{eqnarray}
where $r = \frac{\delta \dt}{\dx^2}$. Dividing through by the common $\epsilon_{m}^{n}$ term we can solve for the error growth factor
\begin{eqnarray}
e^{\alpha \dt} &=& 1-2r+2r\left(\frac{e^{i\beta\dx}+ e^{-i\beta\dx}}{2}\right),\\ \nonumber
&=& 1-4r\left(\frac{1-\cos(\beta\dx)}{2}\right),\\ \nonumber
&=& 1-4r\sin^{2}\left(\frac{\beta\dx}{2}\right),
\end{eqnarray}
reduced using the identities in (\ref{eqn:identities}). Using a tight ($C=0$) von Neumann criterion (\ref{VN criterion}), we bound $|e^{\alpha\dt}|\leq 1$ with the error growth factor in place of the spectral radius. Since the error growth factor is real, the bound is ensured in the two components, $e^{\alpha\dt}\leq 1$ which holds trivially and $e^{\alpha\dt}\geq -1$ which is true at the extremum as long as $r\leq\frac12$. Thus, as long as $dt\leq\frac{\dx^2}{2\delta}$, the Forward Euler method for the Test equation is stable in the sense of the tight von Neumann criterion which ensures diminishing errors at each step.
\end{example}

The balancing of explicit and implicit components in the Crank-Nicolson method create a much less restrictive stability condition.
\begin{exercise}\label{ex:CN_Stability}
Use the von Neumann stability analysis to determine conditions on $\dt,\dx$ to ensure stability of the Crank-Nicolson method (\ref{eqn:CN}) for the Test equation (\ref{eqn:LinearTestIBVP}).

{\it Hint: verify that 
\begin{equation} \label{eqn:growthfactor_CNtest}
e^{\alpha\dt} =\frac{1-2r\sin^{2}\left(\frac{\beta\dx}{2}\right)}{1+2r\sin^{2}\left(\frac{\beta\dx}{2}\right)},
\end{equation}
and show that both bounds are trivially true so that Crank-Nicolson method is unconditionally stable.}
\end{exercise}

The default method in {\tt PDE\_Solution.m} and demonstrated in Figure \ref{fig:f5eig} is the semi-implicit Crank-Nicolson method (\ref{eqn:CN_SI}) for the Fisher-KPP equation (\ref{eqn:FisherIBVP}). If you set $\rho=0$ in {\tt PDE\_Analysis\_Setup.m}, however, Crank-Nicolson method for the Test equation (\ref{eqn:LinearTestIBVP}) is analyzed instead. 

The two-level matrix form of the semi-implicit Crank-Nicolson method is
\begin{eqnarray}\label{eqn:CN_SI_matrix}
{\bf U}^{n+1} &=& M {\bf U}^n + {\bf N},\\ \nonumber
M &=& \left(I - \frac{\dt}{2}D \right)^{-1}\left( I + \frac{\dt}{2}D +\rho\dt \left( 1 - {\bf U}^n \right) \right),\\ \nonumber
{\bf N} &=& \dt\left(I - \frac{\dt}{2}D \right)^{-1}{\bf B}.
\end{eqnarray}
Notice in Figure \ref{fig:f5eig} that all of the eigenvalues are real and they are all bounded between -1 and 1. Such a bound, $|\lambda|<1$, ensures stability of the method based upon the default choice of $\dt,\dx$ step sizes.

\begin{example} \label{ex:CN_SI_Stability}
Use the von Neumann stability analysis to determine conditions on $\dt,\dx$ to ensure stability of the semi-implicit Crank-Nicolson method (\ref{eqn:CN_SI}) for the Fisher-KPP equation (\ref{eqn:FisherIBVP}).

Now we replace each solution term $U_m^n$ in the semi-implicit Crank-Nicolson method (\ref{eqn:CN_SI}) for the Fisher-KPP equation (\ref{eqn:FisherIBVP})
$$
-\frac{r}{2}U_{m-1}^{n+1}+\left(1-r\right)U_{m}^{n+1}-\frac{r}{2}U_{m+1}^{n+1}= \frac{r}{2}U_{m-1}^{n}+\left(1-r+\rho (1-U_{m}^{n}) \right)U_{m}^{n} +\frac{r}{2}U_{m+1}^{n}
$$
with the representative error $\epsilon_{m}^{n}=\hat{\epsilon}e^{\alpha n\dt}e^{im\beta\dx}$ where again $r = \frac{\delta \dt}{\dx^2}$. Again dividing through by the common $\epsilon_{m}^{n}$ term we can solve for the error growth factor 
$$
e^{\alpha\dt}=\frac{1-2r\sin^{2}\left(\frac{\beta\dx}{2}\right) + \dt\rho(1-\tilde{U})}{1+2r\sin^{2}\left(\frac{\beta\dx}{2}\right)}
$$
where constant $\tilde{U}$ represents the extreme values of $U_{m}^{n}$. Due to the equilibrium points $\bar{u}=0,1$ to be analyzed in section \ref{sub:steadystate}, the bound $0\leq U_{m}^{n}\leq 1$ holds as long as the initial condition is similarly bounded $0\leq U_{m}^{0}\leq 1$.
Due to the potentially positive term $\dt \rho (1-\tilde{U})$, the tight ($C=0$) von Neumann criterion fails, but the general von Neumann criterion (\ref{VN criterion}), $|e^{\alpha\dt}|\leq 1+C\dt$, does hold for $C\geq \rho$. With this assumption, $e^{\alpha\dt}\geq -1-C\dt$ is trivially true and $e^{\alpha\dt}\leq 1+C\dt$ results in
$$
\left(\rho(1-\tilde{U})-C\right)\dx^2 -4\delta \sin^{2}\left(\frac{\beta\dx}{2}\right) \leq 2C\delta\dt \sin^{2}\left(\frac{\beta\dx}{2}\right) 
$$
which is satisfied at all extrema as long as $C\geq \rho$ since $\left(\rho(1-\tilde{U})-C\right)\leq 0$. Thus, the semi-implicit Crank-Nicolson method is unconditionally stable in the sense of the general von Neumann criterion, which bounds the growth of error less than an exponential. This stability is not as strong as that of the Crank-Nicolson method for the Test equation but it provides useful management of the error.
\end{example}

\subsection{Oscillatory Behavior}\label{sub:oscillations}

Oscillatory behavior has been exhaustively studied for ODEs \cite{CGraham2016, Gao2015} with much numerical focus on researching ways to dampen oscillations in case they emerge \cite{Britz2003, Osterby2003}. For those wishing to keep their methods oscillation-free, Theorem \ref{thm:nonnegeig} provides sufficiency of non-oscillatory behavior through the non-negative eigenvalue condition (proven, for example, in \cite{Thomas1995}).
\begin{theorem}[Non-negative Eigenvalue Condition] \label{thm:nonnegeig}
A two-level difference scheme (\ref{eqn:two-level_method}) is free of numerical oscillations if all the eigenvalues $\lambda_i$ of the method matrix $M$ are non-negative.
\end{theorem}
Following von Neumann stability analysis from section \ref{sub:stability}, we can use the error growth factors previously computed to determine the non-negative eigenvalue condition for a given method.
\begin{example}
To find the non-negative eigenvalue condition for the semi-implicit Crank-Nicolson method for the Fisher-KPP equation (\ref{eqn:FisherIBVP}), we start by bounding our previously computed error growth factor as $e^{\alpha\dt}\geq 0 $ to obtain
$$
1-2r\sin^{2}\left(\frac{\beta\dx}{2}\right) + \dt\rho(1-\tilde{U})\geq 0
$$
which is satisfied at the extrema, assuming $0\leq\tilde{U}\leq1$, by the condition 
\begin{equation}\label{eqn:nonnegCN_SI}
\frac{\delta\dt}{\dx^2}\leq \frac12
\end{equation}
which happens to be the same non-negative eigenvalue condition for the Crank-Nicolson method applied to the linear Test equation (\ref{eqn:LinearTestIBVP}).
\end{example}

A numerical approach to track numerical oscillations uses a slight modification of the standard definitions for oscillatory behavior from ODE research to identify numerical oscillations in solutions to PDEs \cite{Lakoba2016P1}.

\begin{definition}
A continuous function $u(x,t)$ is oscillatory about $K$ if the difference $u(x,t)-K$ has an infinite number of zeros for $a\leq t<\infty$ for any $a$.
Alternately, a function is oscillatory over a finite interval if it has more than two critical points of the same kind (max, min, inflection points) in any finite interval $[a,b]$ \cite{Gao2015}.
\end{definition}
Using the first derivative test, this requires two changes in the sign of the derivative. Using first order finite differences to approximate the derivative results in the following approach to track numerical oscillations.

\begin{definition}[Numerical Oscillations]
By tracking the sign change of the derivative for each spatial component through sequential steps $t_{n-2}, t_{n-1}, t_{n}$ in time, oscillations in time can be determined by the logical evaluation
$$
(U^{n-2}-U^{n-1}) (U^{n-1}-U^{n}) < 0,
$$
which returns true (inequality satisfied) if there is a step where the magnitude oscillates through a critical point. Catching two such critical points will define a numerical oscillation in the solution.
\end{definition}

Crank-Nicolson method is known to be unconditionally stable, but damped oscillations have been found for large time steps. The point at which oscillations begin to occur is an open question, but it is known to be bounded from below by the non-negative eigenvalue condition, which can be rewritten from (\ref{eqn:nonnegCN_SI}) as $\dt \leq \frac{\dx^2}{2\delta}$. Breaking the non-negative eigenvalue condition, however, does not always create oscillations.
\cproblem{\label{prob:stepconditions}
Using the example code found in Appendix \ref{App:Solutions}, uncomment lines 12-15 (deleting \%'s) and comment lines 7-9 (adding \%'s) to use the semi-implicit Crank-Nicolson method. Make sure the default values of $\dx=0.05, \rho=1$ are set in {\tt PDE\_Analysis\_Setup.m} and choose step size $\dt=2$ in {\tt PDE\_ Solution.m}. Notice how badly the non-negative eigenvalue condition $\dt\leq\frac{\dx^2}{2\delta}$ fails and run {\tt PDE\_Solution.m} to see stable oscillations in the solution. Run it again with smaller and smaller $\dt$ values until the oscillations are no longer visible. Save this point as $(\dx,\dt)$. Change $\dx=0.1$ and choose a large enough $\dt$ to see oscillations and repeat the process to identify the lowest time step when oscillations are evident. Repeat this for $\dx=0.5$ and $\dx=1$, then plot all the $(\dx,\dt)$ points in MATLAB by typing {\tt plot(dx,dt)} where {\tt dx,dt} are vector coordinates of the $(\dx,\dt)$ points. On the Figure menu, click {\it Tools},then {\it Basic Fitting},  and check {\it Show equations} and choose a type of plot which best fits the data. Write this as a relationship between $\dt$ and $\dx$.
}

Oscillations in linear problems can be difficult to see, so it is best to catalyze any slight oscillations with oscillatory variation in the initial condition, or for a more extreme response, define the initial condition so that it fails to meet one or more boundary condition. Notice that the IBVP will no longer have a unique theoretical solution, but the numerical method will force an approximate solution to the PDE and match the conditions as best as it can. If oscillations are permitted by the method, then they will be clearly evident in this process. 
\resproject{
In {\tt PDE\_Analysis\_Setup.m}, set {\tt rho=0} on line 12 and multiply line 16  by zero to keep the same number of elements, \\
{\tt u0 = 0*polyval(polyfit(\dots}\\
Investigate lowest $\dt$ values when oscillations occur for $\dx=0.05,0.1,0.5,1$ and fit the points with the Basic Fitting used in Challenge Problem \ref{prob:stepconditions}. Then, investigate a theoretical bound on the error growth factor (\ref{eqn:growthfactor_CNtest}) for the Crank-Nicolson method to the Test equation which approximates the fitting curve. It may be helpful to look for patterns in the computed eigenvalues at those $(\dx,\dt)$ points.
}

\section{Parameter Analysis}\label{sec:parameters}
Though parameters are held fixed when solving a PDE, varying their values can have an interesting impact upon the shape of the solution. We will focus on how changing parameters values and initial conditions affect the end behavior of IBVPs. For example, Figure \ref{fig:f5steadystatestab} compares two very different steady-state solutions based upon similar sinusoidal initial conditions. Taking the limit of parameters in a given model can help us see between which limiting functions the steady-state solutions tend.


\begin{figure}
\centerline{
\includegraphics[scale=.45]{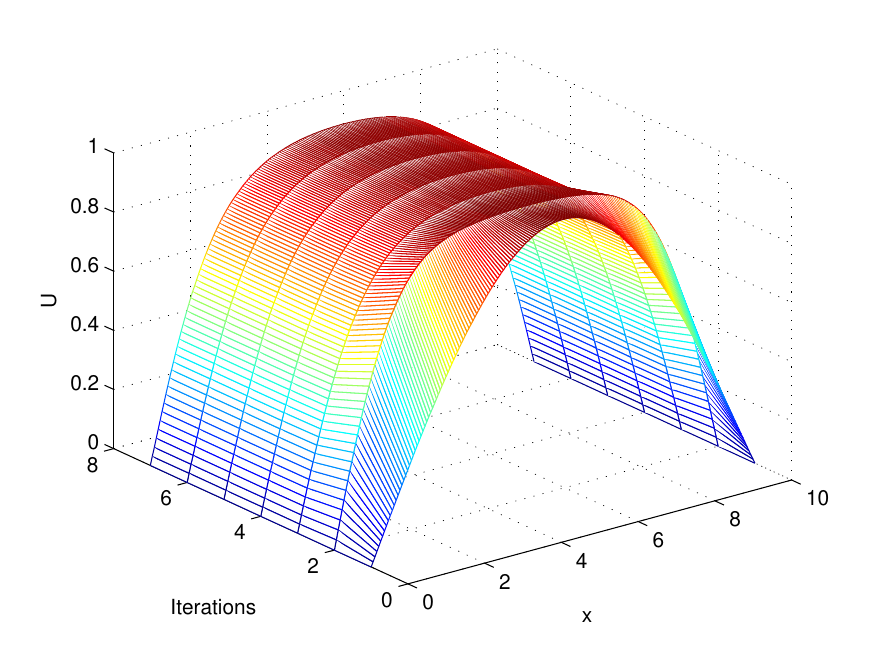} 
\hfil
\includegraphics[scale=.45]{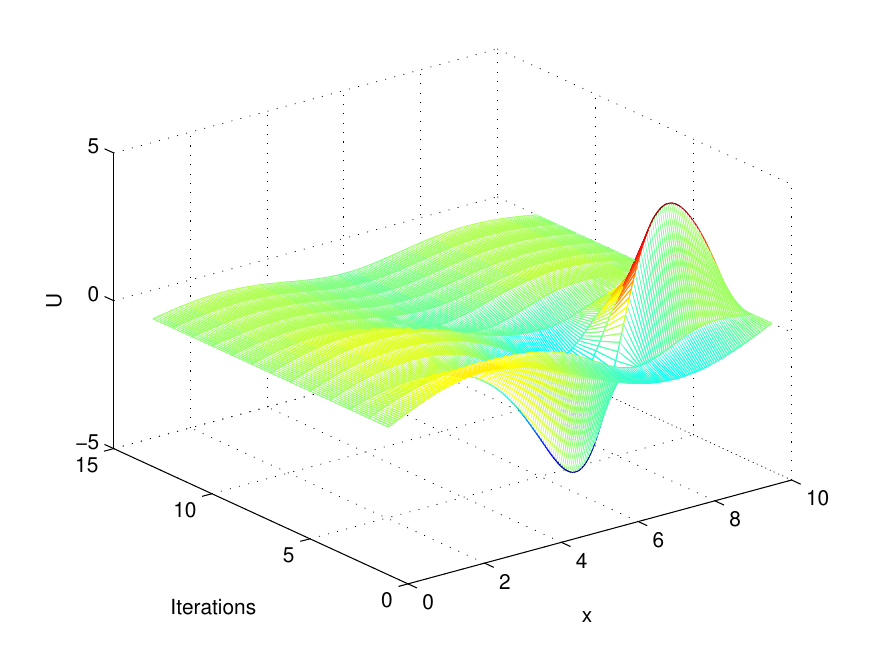} 
}
\caption{Example graphs of Newton's method using the parameter values $a=0, b=0, L=10, \delta = 1, \dx=\frac{1}{20}, \rho=1, {\rm degree}=2, c=1$ and replacing the default initial condition with (left) $\sin\left(\frac{\pi t}{L}\right)$ and (right) $\sin\left(\frac{2\pi t}{L}\right)$ in {\tt PDE\_Analysis\_Setup.m} found in Appendix \ref{App:Setup}.}
\label{fig:f34steadystate}
\end{figure}

\subsection{Steady-State Solutions}
\label{sub:steadystate}

To consider steady-state solutions to the general reaction diffusion model (\ref{eqn:IBVP}), we set the time derivative to zero to create the boundary-value problem (BVP) with updated $u\equiv u(x)$ satisfying

\begin{eqnarray}\label{eqn:BVP}
0 &=& \delta u_{xx} + R(u),\\\nonumber
&&u(0) = a,\\\nonumber 
&&u(L) = b.
\end{eqnarray}

LeVeque \cite{LeVeque2007}, Marangell et. al. \cite{HarleyMarangell2015}, and Aron et.al. \cite{Aron2014} provide insightful examples and helpful guides for analyzing steady-state and traveling wave solutions of nonlinear PDEs. We will follow their guidelines here in our analysis of steady-states of the Fisher-KPP equation (\ref{eqn:FisherIBVP}). Notice that as $\delta\to \infty$, equation (\ref{eqn:BVP}) simplifies as $0 = u_{xx} + \frac1\delta R(u)\to 0=u_{xx}$. Under this parameter limit, all straight lines are steady-states, but only one,
\begin{equation} \label{eqn:f_infty}
f_\infty(x) =\frac{b-a}{L}x+a,
\end{equation}
fits the boundary conditions $u(0,t)=a,\ u(L,t)=b$. On the other hand, as $\delta\to 0$, solutions to equation (\ref{eqn:BVP}) tend towards equilibrium points, $0=R(\bar{u})$, of the reaction function inside the interval $(0,L)$. Along with the fixed boundary conditions, this creates discontinuous limiting functions 
\begin{equation} \label{eqn:f_0}
f_0^{(\bar{u})}(x) = \left\{\begin{array}{ll} a,& x=0\\ \bar{u},& 0<x<L,\\ b,& x=L \end{array} \right.
\end{equation}
defined for each equilibrium point $\bar{u}$.

\begin{example} 
The Fisher-KPP equation (\ref{eqn:FisherIBVP}) reduces to the BVP with updated $u\equiv u(x)$ satisfying 
\begin{eqnarray}\label{eqn:steadystate} 
0 &=& \delta u_{xx} + \rho u(1-u),\\\nonumber
&&u(0) = 0,\\\nonumber 
&&u(10) = 1, 
\end{eqnarray}
As parameter $\delta$ varies, the steady-state solutions of (\ref{eqn:steadystate}) transform from one of the discontinuous limiting functions
\begin{equation} \label{eqn:ex-f_0}
f_0^{(\bar{u})}(x) = \left\{\begin{array}{ll} 0,& x=0\\ \bar{u},& 0<x<10,\\ 1,& x=10 \end{array} \right.
\end{equation}
for equilibrium $\bar{u}=0$ or $\bar{u}=1$, towards the line 
\begin{equation} \label{eqn:ex-f_infty}
f_\infty(x) =\frac{1}{10}x.
\end{equation}
\end{example}

\subsection{Newton's Method}\label{sub:Newtonmethod}
The Taylor series expansion for an analytical function towards a root, that is $0 = f(x_{n+1})$, gives us 
\begin{equation}
0 = f(x_n) + \dx f'(x_n) + O\left(\dx^2\right)
\end{equation}
Truncating $O\left(\dx^2\right)$ terms, and expanding $\dx = x_{n+1}-x_n$, an actual root of $f(x)$ can be approximated using Newton's method \cite{Burden2011}. Newton's iteration method for a single variable function can be generalized to a vector of functions ${\bf G}(u)$ solving a (usually nonlinear) system of equations $0={\bf G}(u)$ resulting in a sequence of vector approximations $\{{\bf U}^{(k)}\}$. Starting near enough to a vector of roots which is stable, the sequence of approximations will converge to it: $\lim_{k\to \infty} {\bf U}^{(k)}={\bf U}_s$. Determining the intervals of convergence for a single variable Newton's method can be a challenge; even more so for this vector version. Note that the limiting vector of roots is itself a discretized version of a solution, ${\bf U}_s=u({\bf x})$, to the BVP system (\ref{eqn:BVP}).
\begin{definition}[Vector Form of Newton's Method]
For system of equations $0={\bf G}(u)$, 
\begin{equation}\label{eqn:Newton}
{\bf U}^{(k+1)} = {\bf U}^{(k)} - J^{-1}({\bf U}^{(k)}) {\bf G}({\bf U}^{(k)})
\end{equation}
where $J({\bf U}^{(k)})$ is the Jacobian matrix, $\frac{\partial {\bf G}_i(u)}{\partial U_j}$, which is the derivative of ${\bf G}(u)$ in $\mathbb{R}^{M\times M}$ where M is the number of components of  ${\bf G}({\bf U}^{(k)})$ \cite{LeVeque2007}.
\end{definition}

Using the standard centered difference, we can discretize the nonlinear BVP (\ref{eqn:BVP}) to obtain the system of equations $0 = {\bf G}\left( {\bf U}^n\right)$
\begin{equation}
0 = \delta D {\bf U}^n + \rho {\bf U}^n\left( 1-{\bf U}^n \right)
\end{equation}

We need an initial guess for Newton’s method, so we will use the initial condition $u_0$ from the IBVP (\ref{eqn:FisherIBVP}). Table \ref{tab:Newton} shows the change in the solution measured by $||{\bf U}^{(k+1)}-{\bf U}^{(k)}||_\infty=||J^{-1}{\bf G}||_\infty$ in each iteration. As expected, Newton's method appears to be converging quadratically, that is $\epsilon^{(k+1)} = O(\epsilon^{(k)})^2$ according to Definition \ref{def:order-conv}. 

\begin{definition}\label{def:order-conv}
Given an iteration which converges, the {\it order of convergence} $N$ is the power function relationship which bounds subsequent approximation errors as
\begin{equation}
\epsilon^{(k+1)} = O\left((\epsilon^{(k)})^N \right).
\end{equation}

\end{definition}

Note that scientific notation presents an effective display of these solution changes. You can see that the powers are essentially doubling every iteration in column 4 of Table \ref{tab:Newton} for the Approximate Error$^a$ in the Fisher-KPP equation. This second order convergence can be more carefully measured by computing
$$
{\rm order}^{(k)} = {\rm round}\left( \frac{\log\left(\epsilon^{(k+1)}\right)}{\log\left(\epsilon^{(k)}\right)} \right)
$$
where rounding takes into account the variable coefficient $C$ in definition \ref{def:order-conv}. For instance, values in column 4 slightly above 2 demonstrate $C<1$ and those slightly below 2 demonstrate $C>1$. Thus it is reasonable to round these to the nearest integer.
This is not evident, however, for the Test equation because the error suddenly drops near the machine tolerance and further convergence is stymied by round-off error.
If we start with a different initial guess ${\bf U}^{(0)}$  (but still close enough to this solution), we would find that the method still converges to this same solution. Newton's method can be shown to converge if we start with an initial guess that is
sufficiently close to a solution. How close depends on the nature of the problem. For more sensitive problems one might have to start extremely close.
In such cases it may be necessary to use a technique such as continuation to find suitable initial data by varying a parameter, for example \cite{LeVeque2007}.

\begin{figure}[t]
\centerline{
\includegraphics[scale=.65]{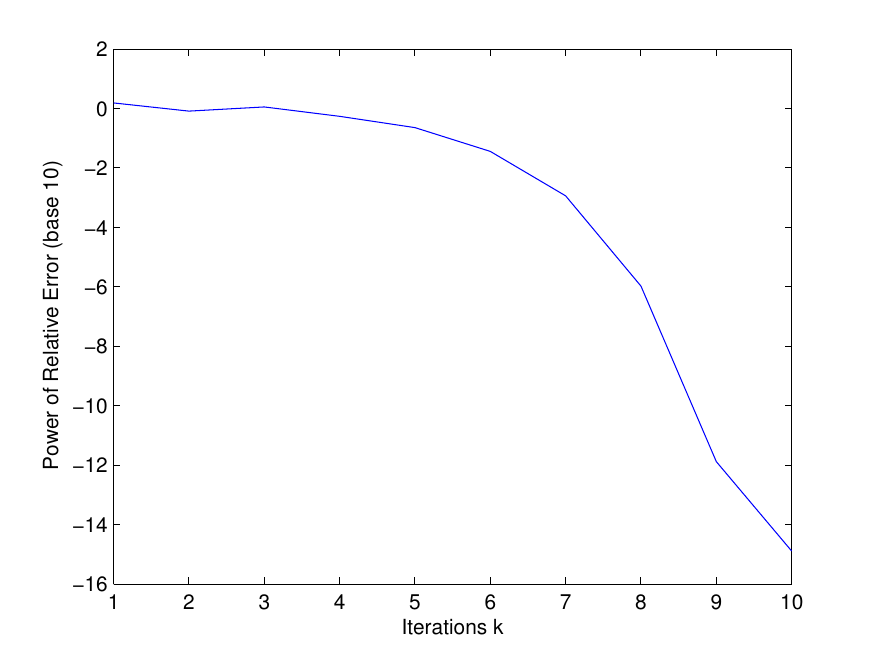} 
}
\caption{Log-plot of approximate relative error,  $\varepsilon^(k)_r=\log_10(\max|{\bf U}^(k+1)-{\bf U}^(k)|)$ in Newton's method for the 
Fisher-KPP equation (\ref{eqn:FisherIBVP})  using the default parameter values $a=0, b=1, L=10, \delta=1, \dx=0.05, \rho=1, {\rm degree}=2, c=\frac{1}{3}$ and initial condition as given in  {\tt PDE\_Analysis\_Setup.m} found in Appendix \ref{App:Setup}. }
\label{fig:lf5convergence}
\end{figure}

\begin{table}
\caption{Verifying Convergence of Newton's Method}
\label{tab:Newton} 
\begin{tabular}{p{1.5cm}p{2.75cm}p{1.75cm}p{2.75cm}p{1.75cm}}
\hline\noalign{\smallskip}
 & \multicolumn{2}{c}{\bf Test Equation} & \multicolumn{2}{c}{\bf Fisher-KPP Equation}  \\
\centering Iteration & \centering Approximate Error$^a$ & \centering Order of Convergence$^b$ & \centering Approximate Error$^a$ & \centering Order of Convergence$^b$ \tabularnewline
\noalign{\smallskip}\svhline\noalign{\smallskip}
1 & 1.6667e-01 & 18.2372 & 1.5421e+00 &0 (-0.4712)\\
2 & 6.4393e-15 & 0.9956 & 8.1537e-01 & -1 (-0.5866)\\
3 & 7.4385e-15 & 1.0585& 1.1272e+00 & -5 (-5.0250)\\
4 & tol.$^c$ reached  & & 5.4789e-01 & 2 (2.4533)\\
5 & & & 2.2853e-01 & 2 (2.2568)\\
6 &  & & 3.5750e-02 & 2 (2.0317)\\
7 &  & & 1.1499e-03 & 2 (2.0341)\\
8 &  & & 1.0499e-06 & 2 (1.9878)\\
9 &  & & 1.3032e-12 & 2 (1.2875)\\
10 & & & tol.$^c$ reached & \\
\noalign{\smallskip}\hline\noalign{\smallskip}
\end{tabular}\\
$^a$ Approximate error measured as the maximum absolute difference ($||\cdot||_\infty$) between one iteration and the next.\\
$^b$ Order of convergence is measured as the power each error is raised to produce the next: $\epsilon_{i+1} = \epsilon_i^p \to p = \log(\epsilon_{i+1})/\log(\epsilon_i)$\\
$^c$ Stopping criterion is reached error is less than tol = $10\epsilon = 2.2204e-15$.


\end{table}


The solution found in Figure \ref{fig:fl6delta_range} for $\delta=1$ is an isolated (or locally unique) solution in the sense that there are no other solutions very nearby. However, it does not follow that this is the unique solution to the BVP (\ref{eqn:steadystate}) as shown by the convergence in Figure \ref{fig:f5steadystatestab} to another steady-state solution. In fact, this steady-state is unstable for the Fisher-KPP equation (\ref{eqn:FisherIBVP}), as demonstrated in Figure \ref{fig:f5steadystatestab}.

 


\begin{figure}[t]
\centerline{
\includegraphics[scale=.65]{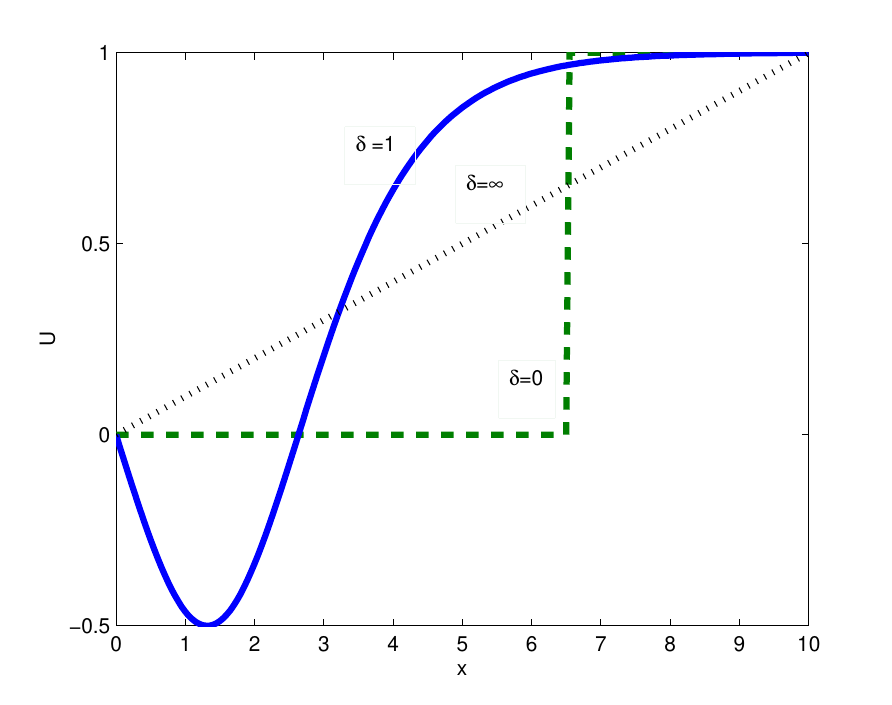} 
}
\caption{Graph of three steady-state solutions to BVP (\ref{eqn:steadystate}) by Newton's method for a $\delta$-parameter range of (1) $\delta=0$, (2) $\delta=1$, and (3) $\delta\to\infty$ using the other default parameter values $a=0, b=1, L=10, \dx=\frac{1}{20}, \rho=1, {\rm degree}=2, c=\frac{1}{3}$ and initial condition as given in {\tt PDE\_Analysis\_Setup.m} found in Appendix \ref{App:Setup}.}
\label{fig:fl6delta_range}
\end{figure}


\begin{svgraybox}
{\bf Project Idea 3.} 
For $\delta=1$, use Newton's method in example code in section \ref{App:Solutions} in the Appendix to investigate steady-state solutions to the bounded Fisher-KPP equation (\ref{eqn:steadystate}). Note the behavior of limiting solutions found in relation to the initial condition used. Also, note the shape of functions for which Newton's method is unstable. One example behavior for a specific quadratic is shown in Figure \ref{fig:fl6delta_range}
\end{svgraybox}

\begin{svgraybox}
{\bf Project Idea 4.} 
Find an initial condition which tends to a steady-state different than either $f_0(x)$ or $f_\infty(x)$. Investigate how the shape of the solution changes as $\delta \to 0$ and as $\delta \to \infty$. Does it converge to a limiting function at both ends?
\end{svgraybox}

\begin{figure}[t]
\includegraphics[scale=.40]{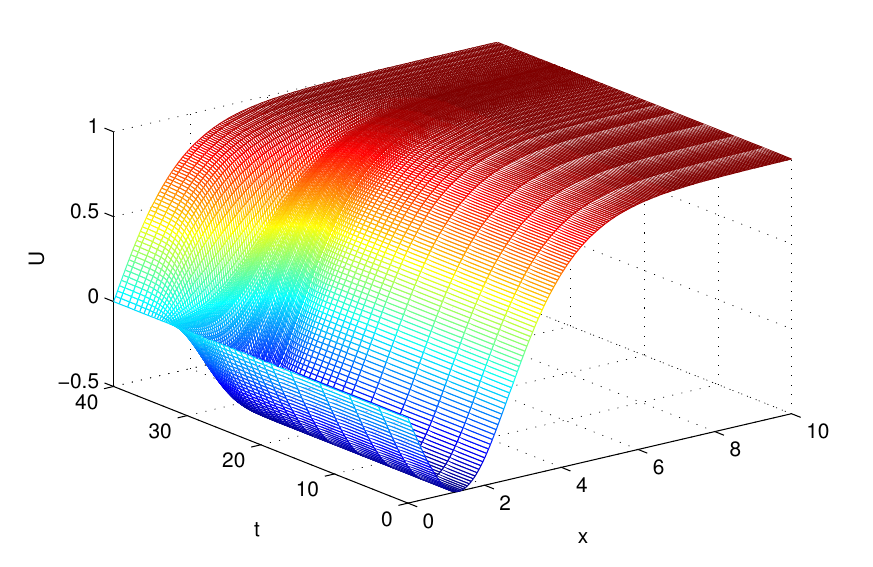} 
\includegraphics[scale=.40]{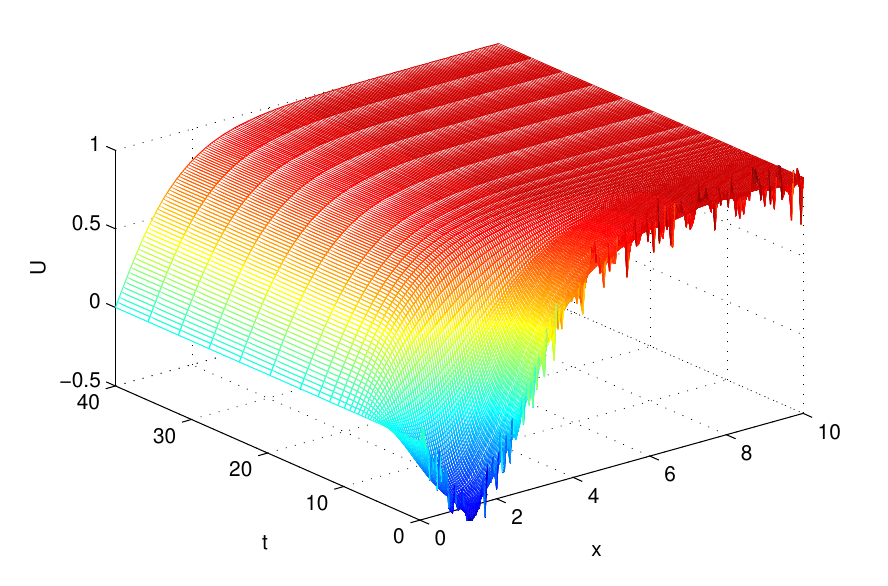} 
\caption{Demonstration of the instability of the steady-state in Figure \ref{fig:fl6delta_range} for $\delta=1$ as an initial condition for the Fisher-KPP equation (\ref{eqn:FisherIBVP}) adding (a) no additional noise and (b) some normally distributed noise to slightly perturb the initial condition from the steady-state using the other default parameter values $a=0, b=1, L=10, \dx=\frac{1}{20}, \rho=1, {\rm degree}=2, c=\frac{1}{3}$ and first initial condition as given in {\tt PDE\_Analysis\_Setup.m} found in Appendix \ref{App:Setup}. }
\label{fig:f5steadystatestab}
\end{figure}

\begin{svgraybox}
{\bf Project Idea 5. } 
Once you find an initial condition which tends to a steady-state different than either $f_0(x)$ or $f_\infty(x)$, run {\tt PDE\_Solution.m} in Appendix \ref{App:Solutions} to investigate the time stability of this steady-state using the built-in solver initialized by this steady-state perturbed by some small normally distributed noise. 
\end{svgraybox}
Note, the {\tt randn(i,j)} function  in MATLAB is used to create noise in {\tt PDE \_Solution.m} using a vector of normally distributed psuedo-random variables with mean 0 and standard deviation 1.

Newton's method is called a local method since it converges to a stable solution only if it is near enough. More precisely, there is an open region around each solution called a basin, from which all initial functions behave the same (either all converging to the solution or all diverging) \cite{Burden2011}.

\begin{svgraybox}
{\bf Project Idea 6. } 
It is interesting to note that when $b=0$ in the Fisher-KPP equation, $\delta$-limiting functions coalesce, $f_\infty(x)=f_0(x)\equiv 0$, for $\bar{u}=0$. Starting with $\delta=1$, numerically verify that the steady-state behavior as $\delta\to\infty$ remains at zero. Though there are two equilibrium solutions as $\delta\to0$, $\bar{u}=0$ and $\bar{u}=1$, one is part of a continuous limiting function qualitatively similar to $f_\infty(x)$ while the other is distinct from $f_\infty(x)$ and discontinuous. Numerically investigate the steady-state behavior as $\delta\to0$ starting at $\delta=1$. Estimate the intervals, called Newton's basins, which converge to each distinct steady-state shape or diverge entirely. Note, it is easiest to distinguish steady-state shapes by number of extremum. For example, smooth functions tending towards $f_0^{(1)}(x)$ have one extrema.
\end{svgraybox}

\subsection{Traveling Wave Solutions}
The previous analysis for finding steady-state solutions can also be used to find asymptotic traveling wave solutions to the initial value problem 
\begin{eqnarray}\label{eqn:IVP}
u_t = \delta u_{xx} + R(u),\\\nonumber
u(x,0) = u_0(x),\ -\infty <x< \infty,\ t\geq 0,
\end{eqnarray}
by introducing a moving coordinate frame: $z=x-ct$, 
with speed $c$ where $c>0$ moves to the right. Note that by the chain rule for $u(z(x,t))$ 
\begin{eqnarray}
u_t = u_z z_t = -c u_z \\\nonumber
u_x = u_z z_x = u_z\\\nonumber
u_{xx} = (u_z)_z z_x = u_{zz}.
\end{eqnarray}
Under this moving coordinate frame, equation (\ref{eqn:IVP}) transforms into the boundary value problem 
\begin{eqnarray}\label{eqn:traveling}
0 &=& \delta u_{zz} +cu_z+ R(u),\\\nonumber
&&
-\infty <z< \infty, 
\end{eqnarray}

\begin{svgraybox}
{\bf Project Idea 7. } 
Modify the example code {\tt PDE\_Analysis\_Setup} to introduce a positive speed starting with $c=2$ (updated after the initial condition is defined to avoid a conflict) and adding {\tt +c/dx*spdiags(ones(M-2,1)*[-1 1 ],[-1 0], M-2, M-2)} to the line defining matrix D and update {\tt BCs(1)} accordingly. Once you have implemented this transformation correctly, you can further investigate the effect the wave speed $c$ has on the existence and stability of traveling waves as analyzed in \cite{HarleyMarangell2015}. Then apply this analysis to investigate traveling waves of the FitzHugh-Nagumo equation (\ref{eqn:FitzHughNagumo}) as steady-states of the transformed equation.
\end{svgraybox}



\section{Further Investigations}\label{sec:further-investigation}
Now that we have some experience numerically investigating the Fisher-KPP equation, we can branch out to other relevant nonlinear PDEs.
\begin{svgraybox}
{\bf Project Idea 8. } 
Use MATLAB's {\tt ode23s} solver to investigate asymptotic behavior of other relevant reaction-diffusion equations, such as the FitzHugh-Nagumo equation (\ref{eqn:FitzHughNagumo}) which models the phase transition of chemical activation along neuron cells \cite{FitzHughNagumo}. First investigate steady-state solutions and then transform the IBVP to investigate traveling waves. Which is more applicable to the model?  A more complicated example is the Nonlinear Schr\"odinger equation (\ref{eqn:NLS}), whose solution is a complex-valued nonlinear wave which models light propagation in fiber optic cable \cite{Lakoba2016P1}.
\end{svgraybox}
For the FitzHugh-Nagumo equation (\ref{eqn:FitzHughNagumo})
\begin{eqnarray} \label{eqn:FitzHughNagumo}
u_t &=& \delta u_{xx} + \rho u(u-\alpha)(1-u),\ 0<\alpha<1\\\nonumber
&&u(x,0) = u_0(x),\\\nonumber
&&u(0,t) = 0\\\nonumber
&&u(10,t) = 1
\end{eqnarray}

For the Nonlinear Schr\"odinger (\ref{eqn:NLS}), the boundary should not interfere or add to the propagating wave. One example boundary condition often used is setting the ends to zero for some large $L$ representing $\infty$.
\begin{eqnarray} \label{eqn:NLS}
u_t &=& i\delta u_{xx} - i\rho |u|^2 u,\\\nonumber
&&u(x,0) = u_0(x),\\\nonumber
&&u(-L,t) = 0\\\nonumber
&&u(L,t) = 0
\end{eqnarray}


Additionally, once you have analyzed a relevant model PDE, it is helpful to compare end behavior of numerical solutions with feasible bounds on the physical quantities modeled. In modeling gene propagation, for example, with the Fisher-KPP equation (\ref{eqn:FisherIBVP}), the values of $u$ are amounts of saturation of the advantageous gene in a population. As such, it is feasible for $0<u<1$ as was used in the stability analysis. The steady-state solution shown in Figure \ref{fig:fl6delta_range}, which we showed was unstable in Figure \ref{fig:f5steadystatestab}, does not stay bounded in the feasible region. This is an example where unstable solutions represent a non-physical solution. While investigating other PDEs, keep in mind which steady-states have feasible shapes and bounds. Also, if you have measurement data to compare to, consider which range of parameter values yield the closest looking behavior. This will help define a feasible region of the parameters.


%
%
%
\appendix
\section*{Appendix}
\addcontentsline{toc}{section}{Appendix}

\subsection{Proof of Method Accuracy}\label{App:AccuracyProof}
Numerical methods, such as those covered in section \ref{sub:methods}, need to replace the underlying calculus of a PDE with basic arithmetic operations in a consistent manner (Definition \ref{def:consistency}) so that the numerical solution accurately approximates the true solution to the PDE. To determine if a numerical method accurately approximates the original PDE, all components are rewritten in terms of common function values using Taylor series expansions (Definition \ref{def:Taylor}). 

\begin{definition}\label{def:Taylor}
Taylor series are power series representations of functions at some point $x=x_{0}+\dx$ using derivative information about the function at nearby point $x_{0}$.
\begin{equation}\label{eqn:cont-Taylor}
f(x)	=f(x_{0})+f'(x_{0})\dx+\frac{f''(x_{0})}{2!}\dx^{2}+...+\frac{f^{(n)}(x_{0})}{n!}\dx^{n}+O\left(\dx^{n+1} \right),
\end{equation} 
where the local truncation error in big-O notation, $\epsilon_L = O\left(\dx^{n+1}\right)$, means $||\epsilon_L|| \leq C\left(\dx^{n+1}\right)$ for some positive $C$ \cite{Burden2011}.
\end{definition}
For a function over a discrete variable ${\bf x}=(x_m)$, the series expansion can be rewritten in terms of indices $U_m \equiv U(x_m)$ as
\begin{equation}\label{eqn:disc-Taylor}
U_{m+1}=U_m+U_m'\dx+\frac{U_m''}{2!}\dx^{2}+...+\frac{U_m^{(n)}}{n!}\dx^{n}+O\left(\dx^{n+1}\right),
\end{equation}

\begin{definition}\label{def:consistency}
A method is {\it consistent} if the difference between the PDE and the method goes to zero as all the step sizes, $\dx, \dt$ for example, diminish to zero. 
\end{definition}
To quantify how accurate a consistent method is, we use the truncated terms from the Taylor series expansion to gauge the order of accuracy  \cite{Burden2011}.
\begin{definition}\label{def:order-accuracy}
The {\it order of accuracy} in terms of an independent variable x is the lowest power $p$ of the step size $\dx$ corresponding to the largest term in the truncation error. Specifically, 
\begin{equation}
{\rm PDE - Method} = O\left(\dx^p\right).
\end{equation}

\end{definition}
\begin{theorem}\label{thm:consistency}
If a numerical method for a PDE has orders of accuracy greater than or equal to one for all independent variables, then the numerical method is consistent. 
\end{theorem}
\begin{proof}
Given that the orders of accuracy $p_1,...,p_k\geq1$ for independent variables $x_1,...,x_k$, then the truncation error
\begin{equation}
{\rm PDE - Method} = O\left(\dx_1^{p_1}\right) +\dots+ O\left(\dx_k^{p_k}\right)
\end{equation}
goes to zero as $\dx_1,...,\dx_k\to 0$ since $\dx_i^{p_i}\to0$ as $\dx_i$ if $p_i\geq 1$.
\end{proof}

\begin{example}(Semi-Implicit Crank-Nicolson Method) \label{ex:CNSI-accuracy-time}
The accuracy in space for the Semi-Implicit Crank-Nicolson method is defined by the discretizations of the spatial derivatives. Following example (ref{ex:discretization}), a centered difference will construct  a second order accuracy in space.
\begin{eqnarray}
&&\frac{\delta}{\dx^2}\left(U_{m-1}^{n}-2U_{m}^{n} +U_{m+1}^{n}\right) -  \delta(U_m^n)_{xx} \\\nonumber
&=& \frac{\delta}{\dx^2}\left(\dx^2(U_m^n)_{xx}+O\left(\dx^4\right)\right) -  \delta(U_m^n)_{xx}\\\nonumber
&=& O\left(\dx^2\right).
\end{eqnarray}
Then the discretized system of equations includes a triadiagonal matrix D with entries $D_{i,i-1}=\frac{\delta}{\dx^2}, D_{i,i}=-\frac{2\delta}{\dx^2}, D_{i,i+1}=\frac{\delta}{\dx^2}$ and a resultant vector ${\bf R}({\bf U}^{n})$ discretized from ${\bf R}(u)$.
The order of accuracy in time can be found by the difference between the method and the spatially discretized ODE system. For the general reaction-diffusion equation (\ref{eqn:IBVP}), 
\begin{eqnarray}\label{eqn:CNSI-accuracy-time}
&&\frac{{\bf U}^{n+1}-{\bf U}^n}{\dt} -  \frac{1}{2} \left(D {\bf U}^{n} + D{\bf U}^{n+1}+{\bf R}({\bf U}^{n}) + {\bf R}({\bf U}^{n+1})\right)\\ \nonumber
&&\hspace{1in}- \left({\bf U}^n_t - D{\bf U}^n -{\bf R}({\bf U}^n) \right)\\ \nonumber
&=& \left( {\bf U}^n_t + O(\dt)\right) - \frac{1}{2} \left(2D {\bf U}^{n} + O(\dt)+2{\bf R}({\bf U}^{n}) + O(\Delta {\bf U})\right)\\\nonumber
&&\hspace{1in} - \left({\bf U}^n_t - D{\bf U}^n -{\bf R}({\bf U}^n) \right)\\ \nonumber
&=& O(\dt)
\end{eqnarray}
since $O(\Delta {\bf U}) =  {\bf U}^{n+1}-{\bf U}^n = \dt {\bf U}^n_t +O(\dt^2)$.
\end{example}

\begin{example}(Improved Euler Method)
Similar to Example \ref{ex:CNSI-accuracy-time}, the accuracy in space for the Improved Euler version of the Crank-Nicolson method (\ref{eqn:CN_IE}) is defined by the same second order discretizations of the spatial derivatives.
The order of accuracy in time can be found by the difference between the method and the spatially discretized ODE system. For the general reaction-diffusion equation (\ref{eqn:IBVP}), 

\begin{eqnarray}\label{eqn:CNIE-accuracy-time}
&&\frac{{\bf U}^{n+1}-{\bf U}^n}{\dt} -  \frac{1}{2} \left(D {\bf U}^{n} + D {\bf U}^*+{\bf R}({\bf U}^{n}) + {\bf R}({\bf U}^*)\right)\\  \nonumber
&&\hspace{1in} - \left({\bf U}^n_t - D{\bf U}^n -{\bf R}({\bf U}^n) \right)\\ \nonumber
&=& \left( {\bf U}^n_t + \frac{\dt}{2}{\bf U}^n_{tt}+ O\left(\dt^2\right)\right) - \frac{1}{2} \left(2D {\bf U}^{n} + \dt\left(D^2 {\bf U}^{n}+D{\bf U}^{n}{\bf R}({\bf U}^{n})\right) \right. \\ \nonumber
&&\hspace{1in} \left.+ O\left(\dt^2\right)+2{\bf R}({\bf U}^{n}) +\left(D {\bf U}^{n}+{\bf R}({\bf U}^{n})\right){\bf R}_u({\bf U}^{n}) + O\left(\dt^2\right)\right) \\ \nonumber
&&\hspace{1in} - \left({\bf U}^n_t - D{\bf U}^n -{\bf R}({\bf U}^n) \right)\\ \nonumber
&=& O\left(\dt^2\right)
\end{eqnarray}
since 
\begin{eqnarray*}
{\bf U}^n_{tt} &=& D^2 {\bf U}^{n} + D{\bf R}({\bf U}^{n})+ \left(D {\bf U}^{n}+{\bf R}({\bf U}^{n})\right){\bf R}_u({\bf U}^{n}),\\
D{\bf U}^* &=& D {\bf U}^n + \dt\left(D^2 {\bf U}^{n} + D{\bf R}({\bf U}^{n})\right),\\
{\bf R}({\bf U}^*) &=& {\bf R}({\bf U}^{n}) + \dt \left(D {\bf U}^{n}+{\bf R}({\bf U}^{n})\right){\bf R}_u({\bf U}^{n}) + O\left(\dt^2\right).
\end{eqnarray*}
\end{example}

\subsection{Main Program for Graphing Numerical Solutions}\label{App:Solutions}
{\scriptsize 
\begin{verbatim}
%% PDE_Solution 
% Numerical Solution with Chosen Solver
PDE_Analysis_Setup;                 % call setup script
tspan = [0 40];                     % set up solution time interval

% Comment when calling Crank-Nicolson method
[t,U] = ode23s(f, tspan, u0);       % Adaptive Rosenbrock Method
N = length(t);
dt = max(diff(t));

% Uncomment to call semi-implicit Crank-Nicolson method
% dt = 0.02;
% t = 0:dt:tspan(end);
% N = length(t);
% U = CrankNicolson_SI(dt,N,D,R,u0,BCs);

 U = [a*ones(N,1),U,b*ones(N,1)];    % Add boundary values
figure();                                   % new figure
mesh(x,t,U);                                % plot surface as a mesh
title('Numerical Solution to PDE');

%% Analyze min/max step ratios and eigenvalues of numerical method matrix
ratio_range = [delta*min(diff(t))/dx^2, delta*max(diff(t))/dx^2],
% Example for semi-implicit CN method for Fisher-KPP and Test equation
Mat=(eye(size(D))-dt/2*D)\(eye(size(D))+dt/2*D+rho*dt*diag(1-U(end,2:end-1))); % update as needed

eigenvalues = eig(Mat);
eig_real_range = [min(real(eigenvalues)),max(real(eigenvalues))],
spectral_radius = max(abs(eigenvalues)),
figure(); plot(real(eigenvalues),imag(eigenvalues),'*'); 
title('Eigenvalues');

%% Call Newton's Root finding method for finding steady state solution
[U_SS,err] = Newton_System(G,dG,u0);
U_SS = [a*ones(length(err),1),U_SS',b*ones(length(err),1)];	% Add boundary values
figure(); plot(x,U_SS(end,:)); 
title('Steady-State Solution by Newton`s Method')
figure(); mesh(x,1:length(err),U_SS);
title('Convergence of Newton Iterations to Steady-State'); 
figure(); plot(log10(err)); %plot(-log10(err));
title('Log-Plot of Approximate Error in Newton`s Method')
order = zeros(length(err));
for i=1:length(err)-1, order(i) = log10(err(i+1))/log10(err(i)); end

%% Investigate stability of Steady-State Solution by Normal Perturbation
noise = 0;
[t,U] = ode23s(f, tspan, U_SS(end,2:end-1) + noise*randn(1,M-2) );
N = length(t);
U = [a*ones(N,1),U,b*ones(N,1)];    % Add boundary values
figure(); mesh(x,t,U);
title('Stability of Steady-State Solution');
\end{verbatim}
}

\subsection{Support Program for Setting up Numerical Analysis of PDE}\label{App:Setup}
{\scriptsize 
\begin{verbatim}
%% PDE_Analysis_Setup
% Setup Script for Numerical Analysis of Evolution PDEs in MATLAB
%  Example: u_t = delta*u_xx + rho*u(1-u), u(x,0) = u0, u(0,t)=a, u(L,t)=b

close all; clear all;   %close figures, delete variables
a = 0; b = 0;           % Dirichlet boundary conditions
L = 10;                 % length of interval
delta = 1/10;           % creating parameter for diffusion coefficient
dx = 1/20;              % step size in space
x=(0:dx:L)';            % discretized interval
M = length(x);          % number of nodes in space
rho = 1;                % scalar for nonlinear reaction term

% Initial condition as polynomial fit through points (0,a),(L/2,c),(L,b)
degree=2; c=1;        % degree of Least-Squares fit polynomial
u0 = polyval(polyfit([0 L/2 L],[a c b],degree),x(2:end-1)); %sin(pi*x(2:end-1)/L); %

% Discretize in space with Finite-Differences: Boundary and Inner Equations
%  (a - 2U_2 + U_3)*(delta/dx^2) = rho*(-U_2 + U_2.^2),
%  (U_{i-1} - 2U_i + U_{i+1})*(delta/dx^2) = rho*(-U_{i} + (U_{i})^2)
%  (U_{M-2} - 2U_{M-1} + b)*(delta/dx^2) = rho*(-U_{M-1} + U_{M-1}.^2)
D = delta/dx^2*spdiags(ones(M-2,1)*[1 -2 1],[-1 0 1], M-2, M-2);
BCs = zeros(M-2,1); BCs(1) = a*delta/dx^2; BCs(end) = b*delta/dx^2; 

% Anonymous functions for evaluting slope function components
f = @(t,u) D*u + BCs + rho*(u - u.^2);  % whole slope function for ode23s
R = @(u) rho*(u - u.^2);                % nonlinear component function
G = @(u) D*u + BCs + rho*(u - u.^2);    % rewritten slope function
dG = @(u) D + rho*diag(ones(M-2,1)-2*u);% Jacobian of slope function
\end{verbatim}
} 

\subsection{Support Program for Running the Semi-Implicit Crank-Nicolson Method}\label{App:CN}
{\scriptsize 
\begin{verbatim}
%% Semi-Implicit Crank-Nicolson Method 
% function [U]=CrankNicolson_SI(dt,N,D,R,u0,BCs)
%  Solving U_t = D*U + R(U) on Dirichlet boundaries vector BCs as
%  (I+dt/2*D)*U^(k+1) = (U^(k)+dt*(1/2*D*U+BCs+R(U^(k))) with time step dt,
%  time iterations N, matrix D, function R, and initial condition vector u0
function [U]=CrankNicolson_SI(dt,N,D,R,u0,BCs)
    sizeD = size(D,1);      % naming computation used twice
    U = zeros(sizeD,N);     % preallocating vector
    U(:,1) = u0(:);         % initializing with u0 as column vector
    for k = 1:N-1
        U(:,k+1) = (eye(sizeD)-dt/2*D)\...
            (U(:,k) + dt*(1/2*D*U(:,k) + BCs + R(U(:,k))));
    end
    U = U';                 % Transpose for plotting 
end
\end{verbatim}
} 

\subsection{Support Program for Verifying Method Accuracy}\label{App:Accuracy}
{\scriptsize 
\begin{verbatim}
%% Method_Accuracy_Verification
% Script for running numerical method for multiple dt step sizes to verify 
% order of accuracy
PDE_Analysis_Setup;             % call setup script
T = 10;
dt = 1; N = 1+T/dt;
Nruns = 10;
TestU = zeros(Nruns,1);
Difference = TestU; Change = TestU; SF = TestU;
midpt = ceil((M-1)/2);
%U = CrankNicolson_SI(dt,N,D,R,u0,BCs);         % Comment for ode23s
options = odeset('InitialStep',dt, 'MaxStep', dt,'AbsTol',max(u0));
[t,U] = ode23s(f, [0 T], u0, options);          % Comment for CN
Err = zeros(Nruns,1); SF = Err; Order = Err;    % Preallocate
fprintf(sprintf('App. Error\t Sig. Figs\t Order\n'));
for k = 1:Nruns-1
    dt=dt/2;  N = 1+T/dt;
    Uold = U;
%	U = CrankNicolson_SI(dt,N,D,R,u0,BCs);      % Comment for ode23s
    options = odeset('InitialStep',dt, 'MaxStep', dt,'AbsTol',max(u0));
    [t,U] = ode23s(f, [0 T], u0, options);      % Comment for CN
	Err(k+1) = norm(U(end,:)-Uold(end,:),Inf)/norm(U(end,:),Inf);
    SF(k+1) = floor(log10(.5/Err(k+1)));
	Order(k) = log2(Err(k)/Err(k+1));
    fprintf(sprintf('%0.5g\t %d\t %0.5f\n', Err(k),SF(k),Order(k)));
end


\end{verbatim}
} 

\subsection{Support Program for Running Newton's Method for ODE Systems}\label{App:Newton}
{\scriptsize 
\begin{verbatim}
%% Newton's Method for Systems of Equations
% function [U,err] = Newton_System(G,dG,u0)
%  Compute vector U satisfying G(U)=0 with Jacobian dG(U) using iterations 
%  U^(k+1) = U^(k) - dG(U^(k))\G(U^(k)), starting with initial guess u0, 
%  and having approximate errors err
function [U,err] = Newton_System(G,dG,u0)
    tol = 10*eps;                   % tolerance relative to machine epsilon
    cutoff = 500;
    err = zeros(cutoff,1);          % Preallocation
    U = zeros(length(u0),length(err));
    U(:,1)=u0;
    for k=1:cutoff
       U(:,k+1) = U(:,k) - dG(U(:,k))\G(U(:,k));    % Newton Iteration
       err(k) = norm(U(:,k+1)-U(:,k),Inf);  % Estimate current error
       if err(k) < tol,             % stopping criterion
           err = err(1:k); U = U(:,1:k);    %truncate to current
           break;   % break from loop
       end
    end
end
\end{verbatim}
}



%
%
%

\end{document}